\documentclass[11pt]{amsart}

\usepackage{rotating,diagrams,mathabx,epsfig,amsthm,amssymb,color}
 \usepackage{tikz}\usetikzlibrary{matrix,arrows}
\usepackage{amsfonts,amsmath}
\newtheorem{theorem}{Theorem}[section]
\newtheorem{prop}[theorem]{Proposition}
\newtheorem{lemma}[theorem]{Lemma}

\newtheorem{cor}[theorem]{Corollary}
\newtheorem{remark}[theorem]{Remark}

\newtheorem*{prop*}{Proposition}

\newtheorem{thm}{Theorem}
\newtheorem{corollary}[thm]{Corollary}
\newtheorem{question}[thm]{Question}

\newcommand{\GGG} [1] {{\boldsymbol \gamma}^{ #1}}
\newcommand{\GGGt} [1] {{\boldsymbol {\tilde{\gamma}}}^{ #1}}

\newcommand{\Q}{\mathbb{Q}}
\newcommand{\MCG}{\mathrm{Aut}}
\newcommand{\slope} [1] {\sigma_{#1}}
\newcommand{\AAA}{\boldsymbol \alpha}
\newcommand{\Ta}{\mathbb{T}_\alpha}
\newcommand{\Tb}{\mathbb{T}_\beta}
\newcommand{\Tg}[1]{\mathbb{T}_{\GGG{#1}}}
\newcommand{\s}{\mathfrak{s}}

\newcommand{\kk}{\mathfrak{k}}

\newcommand{\im}{\mbox{Im}}

\newcommand{\spinc}{{\mbox{spin$^c$} }}
\newcommand{\OS}{{Ozsv\'ath-Szab\'o} }
\newcommand{\OnS}{{Ozsv\'ath and Szab\'o}}
\newcommand{\zee}{\mathbb{Z}}

\newcommand{\cue}{\mathbb{Q}}
\newcommand{\cee}{\mathbb{C}}

\newcommand{\RR}{\mathcal{R}}
\newcommand{\R}{\mathbb{R}}

\newcommand{\J}{{\mathcal J}}

\newcommand{\M}{{\mathcal M}}
\newcommand{\N}{{\mathcal N}}

\newcommand{\F}{{\mathbb F}}
\newcommand{\Z}{{\mathbb Z}}

\newcommand{\CFmct}{\underline{\bf CF}^-}
\newcommand{\CFmc}{{\bf CF}^-}
\newcommand{\HFmct}{\underline{\bf HF}^-}
\newcommand{\HFmc}{{\bf HF}^-}

\newcommand{\ts}{\textstyle}
\newcommand{\ul}{\underline}

\newcommand{\tgr}{\widetilde{gr}}

\newcommand{\coker}{\mbox{\rm coker}}
\newcommand{\x}{\mbox{\bf x}}
\newcommand{\rrr}{\mbox{\bf r}}
\newcommand{\y}{\mbox{\bf y}}
\newcommand{\U}{\mbox{\bf u}}
\newcommand{\V}{\mbox{\bf v}}

\renewcommand{\P}{{\mathcal P}}
\newcommand{\Shat}{\widehat{S}}
\newcommand{\aalpha}{\mbox{\boldmath $\alpha$}}
\newcommand{\bbeta}{\mbox{\boldmath $\beta$}}
\newcommand{\super}[2]{\,{{#1}\hspace{-1.5ex}^{{}^{{#2}}}}\,}

\usepackage{graphicx,amsfonts}

\begin{document}
\title{Floer homology and fractional Dehn twists}
\begin{abstract} We establish a relationship between Heegaard Floer homology and the fractional Dehn twist coefficient of surface automorphisms.  Specifically, we show that the rank of the Heegaard Floer homology of a $3$-manifold bounds the absolute value of the fractional Dehn twist coefficient of the monodromy  of any of its open book decompositions with connected binding.  We prove this by showing that the rank of Floer homology gives bounds for the number of boundary parallel right or left Dehn twists   necessary to add to a surface automorphism to guarantee that the associated contact manifold is tight or overtwisted, respectively.  By examining branched double covers, we also show that the rank of the Khovanov homology of a link bounds the fractional Dehn twist coefficient  of its odd-stranded braid representatives.
\end{abstract}
\author{Matthew Hedden}
\author{Thomas E. Mark}
\maketitle

\section{Introduction}

Let $S$ be a compact oriented 2-manifold with a single boundary component, and $\phi$ a homeomorphism of $S$ fixing its boundary pointwise. The fractional Dehn twist coefficient of $\phi$ is a rational number $\tau(\phi)\in \Q$ that depends only on the isotopy class of $\phi$ rel boundary, and can be understood as a measure of the amount of twisting around the boundary effected by $\phi$ compared to a ``canonical''---e.g., pseudo-Anosov---representative of its (free) isotopy class.  More precisely, consider the image of $\phi$ under the natural map $\MCG(S,\partial S)\rightarrow \MCG(S)$ which drops the  requirement that an isotopy fixes the boundary pointwise.  In this latter group, $\phi$ is isotopic to its Nielsen-Thurston representative; that is, there is an isotopy $\Phi:S\times [0,1]\rightarrow S$ such that $\Phi_0=\phi$ and $\Phi_1$ is either periodic, reducible, or pseudo-Anosov\footnote{As in \cite{kazezroberts}, such a map is called reducible only if is not periodic. Moreover, in the reducible case, after an isotopy rel $\partial S$ we get a subsurface of $S$ to which $\phi$ restricts as a map with periodic or pseudo-Anosov representative: we apply the definition of fractional Dehn twist coefficient to the restriction of $\phi$ to that subsurface.}.  Considering the restriction of $\Phi$ to the boundary, we obtain a homeomorphism:
$$  \Phi_\partial: \partial S \times[0,1]  \rightarrow \partial S\times [0,1] $$ 
defined by $\Phi_\partial(x,t)=(\Phi_t(x),t)$.  The fractional Dehn twist coefficient $\tau(\phi)$ can be defined as the winding number of the arc $\Phi(\theta\times[0,1])$ where $\theta\in \partial S$ is a basepoint\footnote{$\tau(\phi)$ can be defined without  Nielsen-Thurston theory  by lifting $\phi$ to the universal cover and using the translation number of an associated action on a line at infinity \cite{Malyutin}.}.  This would appear only to associate a real number to $\phi$, which could depend on the choice of basepoint and isotopy.  The Nielsen-Thurston classification, however, shows that this winding number is a well-defined rational-valued invariant $\tau(\phi)\in \Q$.   The definition extends easily to surfaces with several boundary circles, in which case there is a  corresponding twist coefficient for each component of the boundary. Here we will be concerned only with the case of connected boundary.

The study of fractional Dehn twist coefficients dates at least from the work of Gabai and Oertel \cite{GO} in the context of essential laminations of 3-manifolds, where, with different conventions than those used here, it appeared as the slope of the  ``degenerate curve" \cite[pg. 62]{GO}. Honda, Kazez, and Matic \cite{HKMVeer1,HKMVeer2} observed a connection with contact topology through open book decompositions, which has been explored by various authors \cite{johnjohn,kazezroberts,ItoKawamuro}. The following proposition summarizes a few key properties of the fractional Dehn twist coefficient. 
\begin{prop*}\cite{Malyutin,ItoKawamuro} Let  $\tau: \MCG(S,\partial S)\to \cue$ be the fractional Dehn twist coefficient, and let $t_\partial$ denote the mapping class of a right-handed Dehn twist around a curve parallel to $\partial S$. Then for all $\phi,\psi\in\MCG(S, \partial S)$, we have:
\begin{enumerate} 
\item {\em (Quasimorphism)}  
\quad  $|\tau(\phi\circ\psi) - \tau(\phi) - \tau(\psi)| \leq 1.$

\item {\em (Homogeneity) }\quad \quad $\tau(\phi^n) = n \tau(\phi).$

\item {\em (Boundary Twisting)} \ \ \
$\tau(\phi\circ t_\partial) = \tau(\phi) + 1.$

\end{enumerate}
\end{prop*}
The first two properties easily imply that the fractional Dehn twist is invariant under conjugation (see e.g., \cite[Proposition 5.3]{Ghys}), and the third 
implies that it can be arbitrarily large, either positively or negatively. There are constraints, however, on the possible denominators of $\tau(\phi)$ based on the topology of $S$; c.f. \cite[Theorem 8.8]{Gabai5}, \cite[Theorem 4.4]{kazezroberts}, \cite{Roberts2001b}.

Surface homeomorphisms of the sort we consider arise naturally as monodromies of fibered knots in 3-manifolds or, equivalently, open book decompositions of 3-manifolds with connected binding. Indeed, if $K\subset Y$ is a fibered knot then the complement of a neighborhood of $K$ is a bundle over $S^1$ with fiber a compact surface $S$ with one boundary component. This bundle is described by a monodromy homeomorphism   $\phi_K: S\rightarrow S$ that is the identity on the boundary and well-defined up to isotopy and conjugation. Hence we can think of the twist coefficient as giving rise to an invariant of fibered knots in 3-manifolds, $K \mapsto \tau(\phi_K)$, where we suppress the  choice of fibration from our notation. Our main result shows that if the 3-manifold is fixed, then there is an {\it a priori} bound on the value of the twist coefficient for any fibered knot in that manifold.

\begin{thm}\label{thm1} Let $Y$ be a closed oriented 3-manifold. Then there exists an integer $N_Y\geq 0$ with the following property: Let $K$ be any fibered knot in $Y$ and let $\phi_K$ denote its monodromy. Then 
\[
|\tau(\phi_K)| \leq N_Y + 1.
\]
\end{thm}
In the case that a knot fibers in many distinct ways, the bound is to be interpreted as stated: regardless of the choice of fiber, the twist coefficient of the resulting monodromy is bounded by a number depending only on $Y$.  To the best of our knowledge, the only situation prior to our theorem in which such a bound was known is for knots in the 3-sphere, in which case work of Gabai \cite{Gabai5} and Kazez-Roberts \cite{kazezroberts} shows that $|\tau(\phi_K)|\le 1/2$.  Their proof relies on the application of thin position, among other things, and  does not extend to other manifolds in an obvious way.  Our proof exploits the connection between twist coefficients and contact topology, and a connection between contact topology and Heegaard Floer homology. Recall that by a construction of Thurston-Winkelnkemper \cite{Thurston}, a fibered knot $K\subset Y$, regarded as an open book decomposition, uniquely determines a contact structure $\xi_K$ on $Y$  (see \cite{Torisu} for uniqueness). It was shown by Honda, Kazez, and Matic that if $\xi_K$ is tight, then $\tau(\phi_K)\geq 0$ \cite[Theorem 1.1 and Propositions 3.1, 3.2]{HKMVeer1}. Using property $(3)$ of $\tau$, we see that to obtain a lower bound on $\tau(\phi_K)$ it suffices to show that there is a number $N_Y$ such that  the monodromy $\phi_K\circ t_\partial^n$ describes a tight contact structure (on a different 3-manifold) for any $n>N_Y$. Therefore Theorem \ref{thm1} is implied by the following.

\begin{thm}\label{thm2} For a closed oriented 3-manifold $Y$, there is an integer $N_Y\geq 0$ with the following property: Let $\xi$ be a contact structure on $Y$, and choose any open book decomposition $(S,\phi)$   that supports $\xi$ and has connected binding. Then for any $n>N_Y$, the open book   $(S,\phi\circ t^n_{\partial S})$ determines a tight contact structure.
\end{thm}

\noindent  Theorem \ref{thm1} follows from Theorem \ref{thm2} by observing that open book decompositions for $Y$ are in bijection with those for $-Y$ under a correspondence induced by inverting monodromies. Thus $N_Y$ in the first theorem  can be taken as the maximum of $N_Y$ and $N_{-Y}$ from the second.

Theorem \ref{thm2} was first observed by \OnS\ in the case that $Y$ is an $L$-space, in which case $N_Y = 0$ \cite[Theorem 1.6]{Contact}.  Indeed, in that paper they ask the following question:

\begin{question}   \cite[pg. 43]{Contact} Given an open book decomposition $(S, \phi)$ for $Y$, what is the minimum $n$ such that $(S,\phi\circ t_{\partial S}^n)$ specifies a  tight contact structure?
\end{question}

 Theorem \ref{thm2} is proved by a generalization of \OnS's argument, and we obtain  a bound depending on the Heegaard Floer homology of $Y$.  Indeed, 
 \begin{equation}\label{Nestimate1}
N_Y = {1\over 2}(\dim_{\F} \widehat{HF}(Y) - |\mathrm{Tor}\ H_1(Y; \zee)|)
\end{equation}
suffices for the Theorems \ref{thm1} and \ref{thm2} above, where $\widehat{HF}(Y)$ denotes the  Heegaard Floer groups of $Y$  with coefficients in $\F=\Z/2\Z$ and $|\mathrm{Tor}\ H_1(Y; \zee)|$ is the number of elements in the torsion submodule of first singular homology.
  Theorem \ref{thm2} can be viewed as an answer to \OnS's question, and Theorem \ref{thm1} as a  geometric interpretation of  the rank of the Heegaard Floer homology groups of a $3$-manifold $Y$: it is a ``speed limit" for fibered knots in $Y$ with respect to the twist coefficient.  Such an interpretation raises the natural question
  
 \begin{question}\label{question2} Does every $3$-manifold contain a ``fast" knot? That is, a fibered knot for which the absolute value of the twist coefficient lies in the interval $[N_Y,N_Y+1]$?
 \end{question}    
 
This question is closely tied to the conjecture that L-spaces are exactly those 3-manifolds without taut foliations.  Indeed, an affirmative answer to Question \ref{question2} would imply this conjecture, by recent work of Kazez and Roberts \cite{kazezrobertsII}.  In a related direction, is perhaps worth pointing out the following  corollary, stated in terms of the {\it reduced} Heegaard Floer homology groups:

\begin{corollary} \label{cor:fiberedgrowth} Let $K\subset Y$ be a fibered knot, and let $\Sigma_n(K)$ denote its $n$-fold branched cyclic cover.  Then  $$ \dim_\F HF^{red}(\Sigma_n(K)) \ge n\cdot|\tau(\phi_K)|-1.$$  In particular, if $K$ has right- (or left-)veering monodromy then all cyclic branched covers over $K$ with sufficiently large order are not L-spaces.
\end{corollary}

\noindent Note that work of Kazez and Roberts could be used to show that high order branched cyclic covers of  fibered knots with right -or left-veering monodromy are not L-spaces, but that their work wouldn't produce the quantitative growth rate of the corollary.  We expect that for most knots the rank of the reduced Floer homology of branched cyclic covers  will grow at least linearly in the order of the cover, but there are examples (for example, the figure eight knot \cite{BaldwinGenusOne}) for which all branched cyclic covers are L-spaces.

Coming back to Question \ref{question2}, we should remark that there are many  useful refinements of $N_Y$ made possible by taking into account further structure on the Floer groups  (see the remarks after the proof of Theorem \ref{basicthm}), and these refinements should be accounted for in making our question precise. For example, since our argument depends only on one \spinc structure at a time, the value 
\begin{equation}\label{Nestimate2}
N_Y = \underset{\s\in \mbox{\scriptsize spin${}^c$}(Y)}\max{1\over 2}( \dim_\F \widehat{HF}(Y,\s) - 1) 
\end{equation}
also suffices for our theorems (here we assume $Y$ is a rational homology sphere for convenience, c.f.\ Remark \ref{Nestimateremark}). In general \eqref{Nestimate2} produces much tighter bounds than \eqref{Nestimate1}, though both give $N_Y = 0$ if $Y$ is an $L$-space. 

If we are given more data about the knot our bound can be sharpened further. To state one such result, recall that an oriented plane field distribution on a closed oriented 3-manifold is determined up to homotopy by two pieces of data: its associated \spinc structure, together with a ``3-dimensional invariant,'' as described by Gompf \cite{Gompf} (ultimately this classification goes back to Pontryagin). Supposing $\xi$ to be a plane field on $Y$ whose \spinc structure $\s_\xi$ has torsion first Chern class, the 3-dimensional invariant is a rational number called the Hopf invariant $h(\xi)$ (see Equation \eqref{hopfinvt} in Section \ref{gradedsec} below). Now whenever a \spinc structure has torsion Chern class, the associated Heegaard Floer homology group carries a rational-valued grading.   The reduced Floer homology groups are finite-dimensional and, in particular, can be nonzero in at most finitely many degrees. Keeping this in mind, the following theorem provides a more precise bound on the twist coefficient of a fibered knot, given the homotopy data of its associated contact structure.

\begin{thm}\label{thm3} Let $\xi$ be a contact structure on $Y$ whose associated \spinc structure $\s_\xi$ is torsion, and let $(S, \phi)$ be any genus $g$ open book supporting $\xi$ with connected binding. Then the  twist coefficient of $\phi$ satisfies
\[
-1 - \dim_\F HF^{red}_{2g-1-h(\xi)}(-Y, \s_\xi) \,\leq\, \tau(\phi)\, \leq\, 1 + \dim_\F HF^{red}_{-h(\xi)}(-Y,\s_\xi).
\]
\end{thm}

A slightly sharper version is given in Corollary \ref{twistboundcor} below, which also applies when $\s_\xi$ is non-torsion. As before, the bounds on twist number come from estimating the number of boundary twists which, when added to the monodromy, is sufficient to obtain a tight contact structure. These estimates, combined with the dependence on the page genus in  Theorem \ref{thm3} and the remark in the previous paragraph, yield  a surprising  corollary.  It indicates that  ``most" open books which support a given contact structure yield a tight structure after adding a single right-handed Dehn twist along the boundary.

\begin{corollary} Let $\xi$ be a contact structure with torsion Euler class, and let $g$ be sufficiently large.  Then for {\em any} genus $g$ open book decomposition  which supports $\xi$ (with connected binding),  adding a single right-handed, boundary-parallel Dehn twist to the monodromy  produces a tight contact structure.
\end{corollary}

\noindent Indeed, we need only choose $g$ large enough that $HF^{red}_{2g-1-h(\xi)}(-Y,\s_\xi) = 0$. Analogously, we have

\begin{corollary} Fix an oriented 3-manifold $Y$.  Then there is a rational number $r_Y\in \Q$ with the following property. Let $\xi$ be any contact structure on $Y$ with $\s_\xi$ torsion and Hopf invariant $h(\xi)\leq r_Y$. Then $\xi$ becomes tight after adding a single right boundary-parallel twist to the monodromy of any supporting open book with connected binding.
\end{corollary}

\noindent In this corollary we simply choose $r_Y$ so that $HF^{red}_{d}(-Y, \s) = 0$ for all torsion \spinc structures $\s$, whenever $d \geq -r_Y$.

In a different direction, our results readily imply a connection between the ``twist number'' of a closed braid in $S^3$ and the reduced Khovanov homology of the  link obtained as its closure.  We thank John Baldwin and Liam Watson for bringing this to our attention.

\begin{thm}\label{braidthm} Let $L$ be a link in $S^3$, and let $\hat{\beta}$ be any closed braid isotopic to $L$ and having an odd number of strands. Then
\[
|\tau(\hat{\beta})| \leq \dim_\F \widetilde{Kh}(-L) - |\det(L)| + 2.
\]
Here $\tau(\hat{\beta})$ is the twist coefficient of $\beta$, viewed as an element in the mapping class group of the disk with $n$ marked points, and  $\widetilde{Kh}$ denotes reduced Khovanov homology.
\end{thm}

The organization of this article is as follows.  In the next section we  give a proof of our main results, Theorems \ref{thm1} and \ref{thm2}, based on a surgery exact triangle for Heegaard Floer homology with twisted coefficients.  Section \ref{proofsec} also contains the proof of Corollary \ref{cor:fiberedgrowth}. In  Section \ref{braidsec}, we spell out the connection between twist numbers and Khovanov homology.  Then in Section \ref{gradedsec} we revisit our proof of Theorems \ref{thm1} and \ref{thm2} to refine our estimates on $N_Y$ and give the proof of Theorem \ref{thm3}, making use of an absolute grading on Heegaard Floer homology by homotopy classes of oriented plane fields on $Y$ due to Gripp and Huang \cite{GH}. In the final section we provide more details on the construction of the twisted surgery triangle that plays a primary role in the proof of our main theorems.

\bigskip
\noindent {\bf Acknowledgements:}  The authors would like to thank the Tokyo Institute of Technology, and Tamas Kalman, Hitoshi Murakami, and Yuanyuan Bao for organizing a conference held there in 2012, where the idea for this project originated.  We are indebted to John Baldwin, Will Kazez, David Krcatovich, Gordana Matic, Olga Plamenevskaya,  and Liam Watson for interesting conversations and valuable input. Matt Hedden was partially supported by NSF grant DMS-0906258,  NSF CAREER grant DMS-1150872, and an Alfred P. Sloan Research Fellowship.  Thomas Mark was partially supported by NSF grants DMS-0905380 and DMS-1309212.


\section{Proof of Theorems \ref{thm1} and \ref{thm2}}\label{proofsec}
\noindent We  work in characteristic two throughout, and let $\F=\Z/2\Z$.  This is 
 for simplicity, and all our arguments could be made with $\Z$ in place of $\F$.  

\vskip0.1in
In this section we prove Theorem \ref{thm2}, from which Theorem \ref{thm1} will follow easily.  More precisely, we show that adding $$N_Y = {\textstyle 1\over 2}(\mathrm{dim}_\F\widehat{HF}(Y) - |\mathrm{Tor}_\zee \ H_1(Y;\zee)|)$$
right-handed Dehn twists to the boundary of any open book decomposition $(S,\phi)$ of $Y$ will produce an open book decomposition for a tight contact structure.  The key observation is that the manifold specified by  $(S,\phi\circ t_\partial^n)$ is homeomorphic to $Y_{-1/n}(K)$, where $K=\partial S$ is the binding of the open book, viewed as a knot in $Y$.  Let $\xi_{n}$ denote the contact structure on $Y_{-1/n}(K)$ induced by $(S,\phi\circ t_\partial^n)$. Our strategy is as follows
\begin{enumerate}
\item  Observe that to show  $\xi_{n}$ is tight, it suffices by \cite[Theorem 1.4]{Contact} to show that its contact invariant $c(\xi_{n})\in HF^+(-Y_{-1/n}(K))$ is not zero.
\item Fit  $HF^+(-Y_{-1/n}(K))$ into an exact triangle of modules over $\F[U]$
\begin{tikzpicture}[>=latex] 
\matrix (m) [matrix of math nodes, row sep=1em,column sep=1em]
{ HF^+(-Y_{-1/n}(K) ; \F)& &  {HF}^+(-Y; \F) \\
& & \\
&  \underline{HF}^+(-Y_0(K); \F[C_n])& \\  };

\path[->,font=\scriptsize]
(m-1-1) edge[->]   (m-1-3)
(m-1-3) edge[->] node[below] {$\ \ \ \ G$}(m-3-2)
(m-3-2) edge[->] node[below] {$F\ \ \ \ $} (m-1-1);
\end{tikzpicture}

\noindent where the bottom term is a twisted version of the Floer homology for zero surgery on the binding, with coefficients in the  group algebra of the cyclic group $C_n=\Z/n\Z$.

\item\label{item3} Show that non-triviality of   $F$, restricted to a particular subgroup  $$\underline{HF}^+(-Y_0(K),\s_{1-g}; \F[C_n]), $$ 
implies $c(\xi_n)\ne 0$. Here $\s_{1-g}$ is the $\spinc$ structure on  $Y_0(K)$ whose Chern class evaluates to $2-2g$ on the fiber and which is cobordant through the surgery cobordism to the $\spinc$ structure associated to $\xi$.

\item  Show that the subgroup from Step $(3)$ is isomorphic to $\F[C_n]$, as an $\F[U]$--module where $U$ acts as zero.  In particular, this group is a vector space of dimension $n$ over $\F$.

\item Conclude, by  Item \ref{item3} and exactness  at $-Y_0$, that $c(\xi_{n})\ne 0$ provided that $$n> \dim_\F \ \coker\ U: HF^+(-Y)\rightarrow HF^+(-Y),$$ and relate $\dim_\F  \coker\ U$ to $N_Y$.

\end{enumerate}

There are two main technical issues involved  in implementing this strategy.  The first  pertains to Steps $(2)$ and $(3)$.  The issue is that while the surgery exact triangle used for Step $(2)$ appears in various places in the literature, neither the definition nor the geometric content of the maps in the triangle is  totally clear.    For Step $(3)$ we need to prove a naturality result for  the contact invariant under the map  $F$. We achieve this  by first relating the maps in the exact triangle to maps on twisted Floer homology groups associated to $2$-handle cobordisms, and then relying on a naturality result for the contact submodule in twisted Floer homology under these latter maps.  In order to achieve this, we establish a general exact triangle satisfied by the (twisted) Floer homologies of certain triples of Dehn filled manifolds using a  well-known ``exact triangle detection lemma".  The above surgery triangle, and indeed all previously known exact triangles satisfied by Heegaard Floer modules of closed three manifolds, can be viewed as specializations.  
So as not to disrupt the flow of the argument, this discussion is postponed to Section \ref{trianglesec}.

The other technical issue is that Step $(4)$ fails  when the fiber surface $S$ has genus one; the relevant summand of $\underline{HF}^+(-Y_0(K),\F[C_n])$ is  infinite dimensional in this case.  To account for this, we  alter our coefficients through the discussion, replacing  $\F$ with a certain Novikov field $\Lambda$, which is the coefficient module for Floer homology perturbed by a $2$-form.  Using    Floer homology perturbed by a $2$-form Poincar{\'e} dual to a meridian of the binding, the case of genus one proceeds exactly as above. 

\subsection{Essentials of the Proof}\label{subsec:proof}
With the general outline of our proof in place, we turn to the details of the argument.  Suppose $W: Z_1\to Z_2$ is a compact oriented cobordism between closed connected oriented 3-manifolds $Z_1$ and $Z_2$.
 For each \spinc structure $\s$ on $W$ there is an induced homomorphism between the Heegaard Floer homology groups of $Z_1$ and $Z_2$. More generally, if $\mathbb{A}$ is a module for the group algebra $\F[H^1(Z_1; \zee)]$, there is a homomorphism in Floer homology with twisted coefficients,
\[
F_{W,\s}^\mathbb{A}: \ul{HF}^+(Z_1, \s_1; \mathbb{A}) \to \ul{HF}^+(Z_2, \s_2; \mathbb{A}\otimes_{\F[H^1(Z_1)]} \mathbb{K}(W)),
\]
where $\s_i = \s|_{Z_i}$, and $\mathbb{K}(W) = \F[\im(H^1(\partial W)\to H^2(W, \partial W))]$ (c.f. \cite[Theorem 3.8]{HolDiskFour}). 

In the case that $W$ consists of a single 2-handle addition along a knot, the  induced homomorphism is defined by counting holomorphic triangles in a suitable Heegaard triple-diagram. Explicitly, suppose that $Z_i$ are described by pointed Heegaard diagrams $(\Sigma, \AAA, \GGG{i}, w)$ such that $(\Sigma, \AAA, \GGG{1}, \GGG{2}, w)$ is an admissible triple diagram describing $W$ and adapted to the knot in the standard way. Then the Floer chain groups for $Z_i$ are generated over the appropriate coefficient modules by intersection points in $\Ta\cap\Tg{i}$, and $F_W^\mathbb{A}$ is the map induced in homology by the chain map 
\[
F_{W}^\mathbb{A}(U^{-j}\cdot \x)  = \sum_{\y\in\Ta\cap \Tg{2}}\sum_{\psi\in\pi_2(\x,\Theta,\y)} \#\M(\psi)\, \mathcal{A}(\psi)\,U^{n_w(\psi) - j}\cdot \y,
\]
where the sum is over homotopy classes of triangles $\psi$ whose associated moduli space $\M(\psi)$ has dimension 0, and $\Theta\in \Tg{1}\cap\Tg{2}$ is a canonical intersection point. Here $\mathcal{A}: \pi_2(\x,\Theta,\y)\to \mathbb{K}(W)$ is an ``additive assignment'' that we now describe in the situations relevant for us; namely, in the case of a  $2$-handle cobordism associated to a ``zero surgery"  (see \cite{HolDiskFour} for more details, or Section \ref{trianglesec} below). 

Assume that $Z_0$ is the $3$-manifold  resulting from 0-framed surgery along a null-homologous knot in a 3-manifold, $Z$, and that $W$ is the associated cobordism.  The oriented boundary of $W$ is given as $$\partial W= -Z \cup Z_0=-(-Z_0)\cup -Z,$$ indicating that we can view $W$ as a cobordism either from $Z$ to $Z_0$, or from $-Z_0$ to $-Z$.  The latter viewpoint will be more relevant for our purposes.  Note that $\mathbb{K}(W)=\F[H^1(Z_0)]$, so that  any choice of coefficient module $\mathbb{A}$ chosen for $-Z_0$ will induce the module $\mathbb{A}\otimes_{\F[H^1(Z_0)]}\mathbb{K}(W)=\mathbb{A}$ for the Floer homology of $-Z$.  

We will primarily specialize to the case where $\mathbb{A}=\F[C_n]$ is the group algebra over $\F$   on the cyclic  group $C_n$, though we will also use coefficients in the group algebra on $C_n$ over the Novikov field $\Lambda$.  For both, suppose that we are given a  Heegaard triple  diagram compatible with the cobordism as above, so that it contains a curve  representing the $0$-framed longitude.  On this curve we place  a basepoint $p$.   Then for $\psi\in \pi_2(\x, \Theta, \y)$  let $n_p(\partial\psi)$ be the algebraic number of times the boundary of $\psi$ meets the codimension-one submanifold (of the Lagrangian  torus) determined by $p$. Taking coefficients in the module $\F[C_n]$, where $C_n$ is the cyclic group of order $n$, the map induced by $W$ can be written 
\begin{equation}\label{cobordmap}
F^{\F[C_n]}_W(\zeta^k\cdot U^{-j}\cdot\x) = \sum_{\y\in\Ta\cap \Tg{1}}\sum_{\psi\in\pi_2(\x,\Theta,\y)} \#\M(\psi)\, \zeta^{n_p(\partial\psi)+k}\cdot U^{n_w(\psi) - j}\cdot \y
\end{equation}
yielding a map on homology
\[
F^{\F[C_n]}_W: \ul{HF}^+(-Z_0; \F[C_n])\to \ul{HF}^+(-Z; \F[C_n]).
\]
When we view $W$ as a cobordism from $Z$ to $Z_0$, then for any coefficient module $\mathbb{A}$ over $\F[H^1(Z)]$, the induced module over $\F[H^1(Z_0)]$ is
$$  \mathbb{A}\otimes_{\F[H^1(Z)]}\mathbb{K}(W)\cong\mathbb{A}\otimes_{\F[H^1(Z)]}\F[H^1(Z_0)]\cong \mathbb{A}[T,T^{-1}],$$
with isomorphisms  induced by the splitting $H^1(Z_0)\cong H^1(Z)\oplus \Z$, and where the additional variable $T$ corresponding to a generator of the $\Z$ summand.
  For $\F$ coefficients, we thus have a chain map
\begin{equation}\label{cobordmap2}
F^{\F}_W( U^{-j}\cdot\x) = \sum_{\y\in\Ta\cap \Tg{1}}\sum_{\psi\in\pi_2(\x,\Theta,\y)} \#\M(\psi)\, T^{n_p(\partial\psi)}\cdot U^{n_w(\psi) - j}\cdot \y
\end{equation}
which induces a map 
\[
F^{\F}_W: {HF}^+(Z)\to \ul{HF}^+(Z_0; \F[T,T^{-1}]).
\]

We will ultimately need to use the map \eqref{cobordmap} in the case that $-Z_0 = -Y_0(K)$ and $-Z = -(Y_{-1/n}(K)) = (-Y)_{1/n}(K)$, where $K$ is the connected binding of an open book in a $3$-manifold $Y$ supporting a contact structure $\xi$ as above.   Note that $Z_0$ is indeed obtained by zero surgery on a knot in $Z$; namely, the core of the surgery solid torus used to obtain $Z$ as $-1/n$ surgery on $K\subset Y$.  Moreover, this  knot is fibered in $Z$ with  monodromy differing from that of $K$ by $n$ right-handed boundary Dehn twists, and thus it induces the contact structure we had been calling $\xi_n$ on $Z=Y_{-1/n}(K)$.   While this is the application we have in  mind, for the moment we suppress the auxiliary 3-manifold,  fibered knot, and contact structure,  $(Y,\xi,K)$, and simply consider the general  case of a fibered knot $L\subset Z$ inducing a contact structure which we abusively denote by $\xi$, and the associated zero surgery $Z_0$.

We will need a generalization of the Heegaard Floer contact invariant introduced in \cite{Contact} to the situation of twisted coefficients.  This generalization is alluded to in \cite[Remark 4.5]{Contact} and further developed in \cite[Section 4]{GenusBounds}.  The construction associates to a contact structure $\xi$ on $Z$ and any module $\mathbb{A}$ over $\F[H^1(Z)]$, a distinguished submodule:
$$ \ul{c}(\xi;\mathbb{A}):=\iota(\mathbb{A})\subset \ul{HF}^+(-Z;\mathbb{A}).$$
This {\em contact submodule} is generated by the inclusion $\iota$ of the homology of the ``bottommost" non-trivial filtered submodule  of the knot Floer homology of a fibered knot supporting $\xi$ (which is isomorphic to $\mathbb{A}$ by \cite[Proof of Theorem 1.1]{Contact})  into  the Floer homology of $-Z$.   Strictly speaking, the literature only refers to the contact {\em element} in twisted Floer homology, but this does not make sense with coefficients in a general module. 

The contact submodule behaves well with respect to the $2$-handle cobordism described above, a fact which we now make precise.  To state the result, note that there is a canonical \spinc structure $\s_{1-g}$ on $Z_0$ determined by
\begin{itemize}
\item $\s_{1-g}$ is cobordant  through the surgery cobordism  to the \spinc structure on $Z$ determined by the contact structure, \item  If $\Shat$ denotes the fiber of the open book, capped off in $Z_0$, then we have:
\begin{equation}\label{eq:extremal}
\langle c_1(\s_{1-g}), [\Shat]\rangle = 2-2g.
\end{equation}
\end{itemize}

\begin{lemma}\label{fulltwistlemma} Let $L\subset Z$ be a fibered knot with induced contact structure $\xi_L$. If the fiber of $L$ has genus greater than one, then for any module $\mathbb{A}$ over $\F[H^1(Z_0(L))]$ there is an identification 
\[
\ul{HF}^+(-Z_0(L),\s_{1-g};\mathbb{A}) \cong  \mathbb{A}\]
as a trivial $\F[U]$-module, i.e. $U$ acts as zero.  Moreover, $\s_{1-g}$ is the unique $\spinc$-structure satisfying \eqref{eq:extremal} and supporting non-zero Floer homology.   The image of the map 
\[
\ul{F}_{W}^{\mathbb{A}}: \ul{HF}^+(-Z_0(L), \s_{1-g};\mathbb{A})\to \ul{HF}^+(-Z, \s_\xi;\mathbb{A})\]
induced by the 0-surgery cobordism is  the contact submodule $\ul{c}(\xi_L;\mathbb{A})$.

All of the above remains true if the genus of the fiber is one, provided that 
we take coefficients in an algebra over $\Lambda_\omega$, where $\Lambda_\omega$ is the Novikov field viewed as a module over $\F[H^1(Z_0)]$ via a choice of  closed 2-form $\omega$ which evaluates non-trivially on the capped-off fiber.  
\end{lemma}


 In the Novikov twisted case, the primary ground algebra for our purposes is  $\Lambda_\omega[C_n]$, where $\omega$ is Poincar{\'e} dual to the class of the fiber.  We refer to Section \ref{trianglesec} for details regarding Novikov ring coefficients. 

\begin{proof} In \cite[Proposition 3.1]{Contact}, \OS construct a Heegaard triple diagram for the surgery cobordism $W: -Z_0(L)\to -Z$ with the following properties:
\begin{itemize}
\item The diagram is weakly admissible for the unique \spinc structure on $W$ extending $\s_{1-g}$.
\item There are precisely two intersection points $\U$, $\V$ providing generators for the chain complex $CF^+(-Z_0(L), \s_{1-g})$, and the only nontrivial differential is $\partial^+(U^{-j}\U) = U^{-j+1}\cdot\V$. Thus $HF^+(-Z_0(L), \s_{1-g})$ is generated by the homology class of $\U$.
\item There is a unique holomorphic triangle $\psi$ contributing to the image of $\U$ under the chain map $F_W$.
\item The image of $\U$ is a cycle representing the contact invariant $c(\xi)\in HF^+(-Z)$. More precisely, the image of $\U$ is the unique generator for the knot Floer chain complex for $L$, in filtration level $-g$.
\end{itemize}
Using the same diagram for the chain complexes and chain maps with twisted coefficients gives the desired result.   Indeed, all of the statements above remain true with $\mathbb{A}$ replacing the implicit $\F$ coefficients.  Note that the second item in this case establishes an isomorphism  $HF^+(-Z_0(L), \s_{1-g};\mathbb{A})\cong \mathbb{A}$ as a trivial $\F[U]$-module, but generation by $\U$ is ambiguous; in particular, it does {\em not} mean generation as an $\F[H^1(Z_0)]$-module since $\mathbb{A}$ may not even be finitely generated over $\F[H^1(Z_0)]$.   This is, in essence, why one needs to talk about the contact {\em submodule} rather than the contact {\em element} in the most general case.

When $g = 1$, the key difference is  that $c_1(\s_{1-g})$ is torsion and the diagram fails to be admissible.  However, it fails to be admissible  only because of  positivity of the periodic domain corresponding to the homology class of the fiber.  If we take  coefficients in an algebra over $\Lambda_\omega$,  where $\omega$ evaluates non-trivially on this class,  then no admissibility is required for this periodic domain.  \end{proof}

Just as with contact element in untwisted Floer homology, non-vanishing of the contact submodule implies  tightness (c.f. \cite[Theorem 1.4]{Contact}).

\begin{lemma} Suppose a contact structure $\xi$ is overtwisted.  Then the contact submodule is trivial, i.e.  $\ul{c}(\xi;\mathbb{A})\equiv 0$ for any ground module $\mathbb{A}$.
\end{lemma}
\begin{proof} This follows in exactly the same manner as \cite[Proof of Theorem 1.4]{Contact}, noting only that the K{\"u}nneth theorem for the knot Floer filtration of the connected sums of knots holds with arbitrary coefficient modules and that the contact submodule associated to the overtwisted contact structure induced by the left-handed trefoil is trivial over any ground module (this latter fact can be calculated directly or deduced from the universal coefficients theorem and the fact that $\F[H^1(S^3)]=\F$).
\end{proof}

The preceding two lemmas yield an immediate corollary.

\begin{cor} In the situation of Lemma \ref{fulltwistlemma}, and for $g\geq 2$, if the map
\[
F^{\F[C_n]}_W: \ul{HF}^+(-Z_0(K), \s_{1-g}; \F[C_n])\to \ul{HF}^+(-Z, \s_\xi; \F[C_n])
\]
is nonzero, then $\xi_L$ is tight. If $g=1$ and the map 
\[
F^{\Lambda_\omega[C_n]}_{W}: \ul{HF}^+(-Z_0(K), \s_{1-g}; \Lambda_\omega[C_n])\to \ul{HF}^+(-Z, \s_\xi; \Lambda_\omega[C_n])
\]
is nonzero, then $\xi_L$ is tight. \qed
\end{cor}

We now return to the surgery exact triangle. In this case, the manifold $Z$ above becomes $Y_{-1/n}(K)$, while $Z_0 = Y_0(K)$. We have:

\begin{prop}\label{tightprop} Let $n>0$, and suppose the genus of the fiber is greater than one. If, for the map $$F: \ul{HF}^+(-Y_0(K), \F[C_n])\to HF^+(-Y_{-1/n}(K); \F)$$ appearing in the surgery triangle, the restriction to the summand corresponding to $\s_{1-g}$ is nontrivial, then the contact structure $\xi_n$ on $Y_{-1/n}$ is tight.   The  same is true if the genus of the fiber is one, provided we consider the surgery triangle  with coefficients in the  Novikov module $\Lambda_\omega$ associated to a $2$-form evaluating non-trivially on the fiber.
\end{prop}

\begin{proof} In Section \ref{trianglesec} the map $F$ appearing in the surgery triangle is defined on the chain level by
\[
F(U^{-j}\zeta^k\x):= \sum_{\stackrel{\scriptstyle\psi\in \pi_2(\x,\Theta,\y)}{\stackrel{{\mu(\psi)=0}}{{{n_p(\partial\psi)=-k \ \text{mod} \ n}}}}}\#{\mathcal{M}}(\psi)\cdot t^{\omega(\psi)}\cdot U^{n_w(\psi)-j}\cdot\y
\]
(c.f. equation \eqref{f2def} below). Here $t$ is the variable appearing in the Novikov ring $\Lambda_\omega$, which we set equal to 1 if the genus of the fiber is at least two. Comparing this to the definition of the cobordism-induced homomorphism $F^{\F[C_n]}_W$ in \eqref{cobordmap} above, it is easy to check that there is a commutative diagram
\[
\begin{diagram}
\underline{CF}^+(-Y_0; \F[C_n]) & \rTo^{{F}_W^{\F[C_n]}} & \underline{CF}^+(-Y_{-1/n}; \F[C_n]) \\
\uTo^{\mathcal N} & \rdTo^{F} & \dTo^{\Pi}\\
CF^+(-Y_0; \F) &\rTo^{F_W}& CF^+(-Y_{-1/n}; \F)
\end{diagram}
\]
where $\mathcal N$ and $\Pi$ are chain maps induced by the coefficient $\F$-homomorphisms $\N: \F\to \F[C_n]$ and $\Pi: \F[C_n]\to \F$ given by
\[
\N(1) = \sum_{k=0}^{n-1} \zeta^k \quad\mbox{and}\quad \Pi(p(\zeta)) = p(0),
\]
where $p(\zeta)$ denotes a polynomial in $\zeta$. In particular, we see that if $F$ induces a nontrivial map in homology (in a particular \spinc structure), then so does ${F}_W^{\F[C_n]}$. An analogous diagram exists with $\Lambda_\omega$ replacing $\F$ throughout.   The result follows now from the previous corollary.

\end{proof}

Our main results (Theorems \ref{thm1} and \ref{thm2}) now follow easily. We give a combined restatement.

\begin{theorem}\label{basicthm} Let $Y$ be a closed oriented 3-manifold. Then there exists a constant $N_Y\geq 0$ with the following property. Let $K\subset Y$ be a fibered knot with monodromy $\phi_K$, and let $\tau_K = \tau(\phi_K)$ be the fractional Dehn twist coefficient of $\phi_K$.  Then 
\begin{enumerate}
\item For all $n> N_Y$, the contact structure $\xi_n$ supported by the open book with monodromy obtained from $\phi_K$ by composition with $n$ boundary-parallel Dehn twists is tight.
\item We have the bound
\[
|\tau_K| \leq N_Y + 1.
\]
\end{enumerate}
Moreover, we can take 
\begin{equation}\label{Ndef}
N_Y = {\textstyle 1\over 2}(\mathrm{dim}_\F\widehat{HF}(Y) - |\mathrm{Tor}_\zee \ H_1(Y;\zee)|).
\end{equation}
\end{theorem}

Here $\widehat{HF}(Y)$ indicates the sum of Heegaard Floer groups over all \spinc structures on $Y$, while $|\mathrm{Tor}_\zee\ H_1(Y;\zee)|$ is the order of the torsion subgroup of the first homology.

\begin{proof}
We treat  the case that the fiber genus is at least two explicitly; the genus one case is exactly the same, using coefficients in $\Lambda_\omega$, where $\omega$ is Poincar{\'e} dual to the meridian of the binding.  Strictly speaking, this latter argument produces \eqref{Ndef} with $\mathrm{dim}_\Lambda \widehat{HF}(Y;\Lambda_\omega)$ in place of $\mathrm{dim}_\F\widehat{HF}(Y)$.  These dimensions are equal, however, since $\widehat{HF}(Y;\Lambda_\omega)\cong \widehat{HF}(Y)\otimes_\F \Lambda$ by the universal coefficient theorem and the fact that $[\omega]=0\in H^2(Y;\R)$.  

Proceeding, then, consider the surgery exact triangle:

\begin{tikzpicture}[>=latex] 
\matrix (m) [matrix of math nodes, row sep=1em,column sep=1em]
{ HF^+(-Y_{-1/n}(K) ; \F)& &  {HF}^+(-Y; \F) ,\\
& & \\
&  \underline{HF}^+(-Y_0(K); \F[C_n])& \\  };

\path[->,font=\scriptsize]
(m-1-1) edge[->]   (m-1-3)
(m-1-3) edge[->] node[below] {$\ \ \ \ G$}(m-3-2)
(m-3-2) edge[->] node[below] {$F\ \ \ \ $} (m-1-1);
\end{tikzpicture}

\noindent The summand of the bottom module corresponding to $\s_{1-g}$ is, according to Lemma \ref{fulltwistlemma}, isomorphic to $\F[C_n]$ and therefore of dimension $n$ over $\F$. By Lemma \ref{fulltwistlemma} again and $U$-equivariance of the sequence, the component of $G$ mapping into the $\s_{1-g}$ summand factors through the cokernel of the action of $U$ on $HF^+(-Y;\F)$, which is finite-dimensional and independent of $n$. Hence for $n$ sufficiently large, we conclude $F$ is nonzero and therefore the contact structure $\xi_n$ on $Y_{-1/n}$ is tight by Proposition \ref{tightprop}. 

To estimate the size of $n$ required, observe that it suffices that $n$ be larger than the dimension of the cokernel of $U$, acting on $HF^+(-Y)$. In a given \spinc structure it is easy to see that $\dim \widehat{HF}(Y,\s) = \dim \ker U + \dim \coker U = 2\dim \coker U + k_\s$, where $k_\s$ is the rank of $HF^\infty(Y,\s)$ as a module over $\F[U,U^{-1}]$.
Note that $k_\s$ is $0$ if $\s$ is non-torsion, and at least 1 in the torsion case (c.f. \cite[Theorem 10.1]{HolDiskTwo} and \cite[Lemma 2.3]{OSSymp}). Adding over all \spinc structures, if we set $N_Y = {1\over 2}(\dim \widehat{HF}(Y) - |\mathrm{Tor}_\zee H_1(Y)|)$ it follows that adding at least $N_Y+1$ right twists to the monodromy of any open book with connected binding will produce a tight contact structure.

According to Honda-Kazez-Matic  \cite[Theorem 1.1 and Propositions 3.1, 3.2]{HKMVeer1}, the fractional Dehn twist coefficient of the monodromy of an open book supporting a tight contact structure is nonnegative. Since we added $N_Y+1$ right twists, the new monodromy is $\phi_K \circ t_{\partial}^{N_Y+1}$. Hence
\[
0\leq \tau(\phi_K \circ t_{\partial}^{N_Y+1}) = \tau_K + N_Y +1,
\]
which gives half the desired inequality. For the other half, replace $\phi_K$ by $\phi_K^{-1}$. This amounts to reversing the orientation on $Y$, giving a lower bound on $\tau(\phi_K^{-1}) = - \tau(\phi_K)$ in terms of $N_{-Y}$. But since \eqref{Ndef} is insensitive to the orientation of $Y$ (c.f. \cite[Proposition 2.5]{HolDiskTwo}), the result follows.
\end{proof}

\begin{remark}\label{Nestimateremark}: We could, by the argument in the proof, take 
\[
N_Y = {1\over 2} \max_\s (\dim \widehat{HF}(Y,\s) - k_\s),
\]
which gives \eqref{Nestimate2} in the case that $Y$ is a rational homology sphere. It follows from the main result of \cite{triplecups} together with work of Lidman \cite{lidman} that if $\s$ is a torsion \spinc structure on a 3-manifold with positive first Betti number, there is an estimate
\[
k_\s\geq L_Y :=\left\{ \begin{array}{ll} 2\cdot 3^{(b_1(Y) -1)/2} & \mbox{if $b_1(Y)$ is odd} \\ 4\cdot 3^{b_1(Y)/2 - 1} & \mbox{if $b_1(Y)$ is even.}\end{array}\right.
\]
Thus, letting $L(\s) = L_Y$ if $\s$ is torsion and $L(\s)= 0$ otherwise,  we can take
\[
N_Y = {1\over 2}\max_\s( \dim\widehat{HF}(Y,\s) - L(\s)).
\]
\end{remark}

\begin{remark} If $Y$ is an $L$-space, meaning that $Y$ is a rational homology sphere with $\dim \widehat{HF}(Y) = |H_1(Y;\zee)|$, then the theorem says $|\tau(\phi_K)| \leq 1$ for any fibered knot $K$ in $Y$. In fact, we must have
\[
|\tau(\phi_K)| < 1 \quad\mbox{for any fibered $K$ in an $L$-space}.
\]
Indeed, if $K\subset Y$ has $|\tau(\phi_K)|\geq 1$ then $Y$ admits a taut foliation, according to \cite[Theorem 1.1]{HKMVeer2}.  On the other hand,  $L$-spaces do not admit taut foliations by \cite[Theorem 1.4]{GenusBounds}. Note that while $|\tau(K)|\leq 1/2$ for knots in $S^3$, this is not true for knots in arbitrary $L$-spaces.

\end{remark}

In a similar spirit, we turn to the proof of Corollary \ref{cor:fiberedgrowth}:

\begin{proof}[Proof of Corollary \ref{cor:fiberedgrowth}]  Suppose  that $K$ is a fibered knot in a 3-manifold $Y$ with monodromy $\phi_K$.  Then the $n$-fold cyclic branched cover $\Sigma_n(K)$ is well-defined (in general, it depends on a homomorphism of the knot group to $\Z/n\Z$, but this is specified by counting intersections with the fiber) and has an open book decomposition with the same fiber and monodromy $\phi_K^n$.  The proof of Theorem \ref{basicthm} shows that $$|\tau(\phi_K^n)| \le 1+\dim_\F \mathrm{Coker} \ U : HF^+(\Sigma_n)\rightarrow HF^{+}(\Sigma_n).$$
Homogeneity of the twist coefficient shows that the left-hand side equals $n\cdot |\tau(\phi_K))|$ whereas $1+\dim HF^{red}(\Sigma_n)$ is at least as large as the right-hand side, since reduced Floer homology is defined as the limit (see \cite[Definition 4.7]{HolDisk}) $$ HF^{red}:= \underset{k\rightarrow \infty}{\mathrm{lim}} \mathrm{Coker}\ U^k.$$ \end{proof}

\section{Application to Braids}\label{braidsec}

Fractional Dehn twist coefficients can be defined in various contexts; for example Malyutin \cite{Malyutin} gave a definition of a ``twist number'' for closed braids by considering a braid as an element of the mapping class group of a punctured disk (see also Ito-Kawamuro \cite{ItoKawamuro}). For a braid $\beta$ write $\tau(\beta)\in \cue$ for the twist number; note that while $\tau(\beta)$ is conjugation-invariant and so depends only on the closure $\hat{\beta}$, it is not an invariant of the link type of $\hat{\beta}$. That is to say, different closed braid representatives of a given link will have different twist numbers.

The following application of Theorem \ref{basicthm} was pointed out to us by John Baldwin and Liam Watson.

\begin{theorem}\label{braidthm} Let $L$ be a link in $S^3$, and let $\hat{\beta}$ be any closed braid isotopic to $L$ and having an odd number of strands. Then
\[
|\tau({\beta})| \leq \dim_\F \widetilde{Kh}(-L) - |\det(L)| + 2.
\]
\end{theorem}

Here $\widetilde{Kh}(-L)$ denotes the reduced Khovanov homology of the mirror of $L$, with coefficients in $\F$. 

\begin{proof} 
If $\hat{\beta}$ is a closed braid with axis the unknot $U$ and representing the link type $L$, then forming the double cover of $S^3$ branched along $\hat{\beta}$ gives rise to a 3-manifold $\Sigma_2(L)$ equipped with an open book decomposition lifting the decomposition of $S^3$ with disk pages and binding $U$. Since $\beta$ has an odd number of strands, this open book structure on $\Sigma_2(L)$ has connected binding.

If $\phi$ denotes the monodromy of the lifted open book, then it is not hard to check that $\tau(\phi) = {1\over 2} \tau(\beta)$. From Theorem \ref{thm1} and using \eqref{Nestimate1}, we have
\[
|\tau(\phi)| \leq {1\over 2}(\dim\widehat{HF}(\Sigma_2(L)) - |H_1(\Sigma_2(L))|) + 1.
\]
Now recall that there is a spectral sequence whose $E_2$ page is the reduced Khovanov homology $\widetilde{Kh}(-L)$, and which converges to $\widehat{HF}(\Sigma_2(L))$ (see \cite{Branched}). Hence
\[
\dim \widehat{HF}(\Sigma_2(L))\leq \dim\widetilde{Kh}(-L).
\]
Moreover, $|H_1(\Sigma_2(L))| = |\det(L)|$ unless the latter quantity is 0. Combining these observations gives the desired result.
 
\end{proof}

Suppose that $\beta'$ is an alternating braid on $n\geq 3$ strands. Using Corollary 5.5 of \cite{Malyutin}, for example, it is easy to see that $\tau(\beta') = 0$. However, it is not the case that any braid Markov equivalent to $\beta'$ has vanishing twist number. The following shows that nevertheless, there is an upper bound on the twist number of a braid representing the same link type as $\hat{\beta'}$.

\begin{cor} Let $L$ be an alternating link, and $\beta$ any braid on an odd number of strands with the property that the closure $\hat{\beta}$ is isotopic to $L$. Then
\[
|\tau({\beta})| < 2.
\]
\end{cor}
\begin{proof}
Since $\hat{\beta}$ is the alternating link $L$, the double branched cover $\Sigma_2(L)$ is an $L$-space, and moreover the spectral sequence from $\widetilde{Kh}(-L)$ to $\widehat{HF}(\Sigma_2(L))$ collapses (c.f. \cite{Branched}). The proof of Theorem \ref{braidthm} then gives $|\tau(\hat{\beta}')|\leq 2$, and the strict inequality follows from the remark at the end of Section \ref{proofsec}.
\end{proof}

The corollary applies equally, of course, to braids whose closures are quasi-alternating in the sense of \OS \cite{Branched}: the double cover of $S^3$ branched along such a link is also an $L$-space.

\section{Graded Refinement}\label{gradedsec}

Here we provide the proof of Theorem \ref{thm3}. To do so, we recall absolute gradings constructed on Heegaard Floer homology by Ozsv\'ath and Szab\'o \cite{HolDiskFour} (in the case of torsion \spinc structures)  and Gripp and Huang \cite{GH} (in general). We also clarify some properties of the general grading by homotopy classes of plane fields in the latter case.

\subsection{Absolute Gradings}

Recall that Heegaard Floer homology carries a relative cyclic grading in each \spinc structure. Gripp and Huang \cite{GH} proved that this can be lifted to an absolute grading on $HF^\circ(Y)$ by the set of homotopy classes of oriented 2-plane fields on $Y$, which we will denote by $\P(Y)$ (here $HF^\circ$ indicates any of the versions of Heegaard Floer homology). Note that if $\xi$ is an oriented plane field, then there is an associated \spinc structure $\s_\xi$. In general, $\P(Y)$ is a $\zee$-set whose orbits correspond bijectively to \spinc structures; the orbit corresponding to a given \spinc structure $\s$ is isomorphic to $\zee/d(\s)\zee$, where $d(\s)$ is the divisibility of $c_1(\s)$, see e.g., \cite[Lemma 2.3 and Section 5(i)] {KMcontact}. We denote the $\zee$ action on $\P(Y)$ by $[\xi]\mapsto [\xi] + n$ for $n\in \zee$. Gripp and Huang proved certain properties of the grading on Floer homology by $\P(Y)$, notably:
\begin{itemize}
\item For any plane field $\xi\in \P(Y)$, we have $HF^+_{[\xi]}(Y)\subset HF^+(Y, \s_\xi)$.
\item The grading by plane fields lifts the relative grading on Heegaard Floer homology defined by Ozsv\'ath and Szab\'o. Moreover, if $c_1(\s_\xi)$ is torsion, the summand $HF^\circ_{[\xi]}(Y)$ coincides with the $\cue$-graded summand $HF^\circ_{h(\xi)}(Y, \s_\xi)$, where $h(\xi)$ is the ``Hopf invariant'' of $\xi$, defined as follows. Choose a compact almost-complex 4-manifold $(W, J)$ with $\partial W = Y$ and $TY\cap J(TY)$ homotopic to $\xi$. Then 
\begin{equation}\label{hopfinvt}
h(\xi) = {1\over 4}(c_1^2(W, J) - 3\sigma(W) - 2 \chi(W) + 2),
\end{equation}
where $\sigma$ is the signature of $W$, $\chi(W)$ is the Euler characteristic, and $c_1^2$ denotes the rational-valued square of the Chern class.
\item If $\xi$ is a contact structure on $Y$, then the contact invariant $c(\xi)$ lies in $HF^+_{[\xi]-1}(-Y)$ (where we think of $\xi$ as also defining an oriented plane field on $-Y$; see below for further discussion).   
\end{itemize}
To describe the final property, we say plane fields $p_1$ and $p_2$ are {\em related by} $W$ if there is an almost-complex structure $J$ on $W$ such that the fields of complex tangencies $TY_i\cap J(TY_i)$ represent $p_i$, for $i = 1,2$.
\begin{itemize}

\item Let $W: Y_1\to Y_2$ be a cobordism, and let $p_i \in \P(Y_i)$ for $i = 1,2$. If $x\in HF^+_{p_1}(Y_1)$ is a homogeneous element such that $F_W(x)$ has a nonzero component in $HF^+_{p_2}(Y_2)$, then $p_1$ and $p_2$ are  related by $W$.
\end{itemize}

We will need a slightly more concrete understanding of the $\zee$-action on the set of homotopy classes of oriented plane fields on a 3-manifold. For this, fix a metric on a 3-manifold $M$ and a trivialization $\tau:TM\cong M\times \mathbb{R}^3$; then, by taking normal vectors, an oriented plane field $\xi$ corresponds to a map $\xi^\tau: M\to S^2$. By the Pontryagin-Thom construction homotopy classes of such maps naturally correspond to cobordism classes of framed links in $M$. The action of $n\in \zee$ on $[\xi]$ is given by adding $n$ to the framing of some component of a link representing $[\xi]$. Note that this action implicitly depends on the orientation of $M$.

Observe that the set $\P(Y)$ carries a natural involution $[\xi]\mapsto [\xi]^*$ induced by reversing the orientation of the plane field $\xi$. We claim that this respects the $\zee$-action: indeed, letting $\xi$ correspond to a map to $S^2$ as above, reversing orientation is equivalent to replacing $\xi^\tau$ by $-\xi^\tau$, which clearly can be realized on the level of links by reversing orientation. Moreover, adding a twist to the framing of an oriented link $L$ also adds a (positive) twist to $-L$, so that $([\xi] + n)^* = [\xi]^* + n$. In fact we have: 

\begin{lemma}\label{conjlemma} For any oriented 3-manifold $Y$ with oriented plane field $[\xi]$, there is an isomorphism
\[
HF^+_{[\xi]}(Y, \s_\xi) \cong HF^+_{[\xi]^*}(Y, \s_{\xi^*}).
\]
\end{lemma}
This can be seen as a refinement of the conjugation invariance of Floer homology, since the \spinc structure associated to $[\xi]^*$ is conjugate to $\s_{\xi}$. Given Gripp and Huang's construction, the proof is routine and based on the observation that if $\x$ is a Heegaard Floer generator coming from the diagram $(\Sigma, \aalpha, \bbeta, z)$ and with corresponding \spinc structure $\s$, then the same intersection point interpreted in $(-\Sigma, \bbeta, \aalpha, z)$ corresponds to the conjugate of $\s$ and gives rise to the negative of the gradient-like vector field originally determined by $\x$ (c.f. \cite{GH}, \cite[Theorem 2.4]{HolDiskTwo}). 

\subsection{Graded Cobordism Maps and the Contact Invariant}
It will be useful for us to explore the sense in which cobordism-induced homomorphisms are graded by recalling some of the homotopy classification of almost-complex structures on 4-manifolds. First, observe that an almost-complex structure on an oriented 4-manifold $W$ is the same as a lift of the the classifying map for $TW$ from $BSO(4)$ to $BU(2)$. Since the fiber of the bundle $BU(2)\to BSO(4)$ is $SO(4)/U(2) \cong S^2$, we are interested in obstruction theory for a bundle over $W$ with fiber $S^2$. (An alternate perspective is given by choosing a metric on $W$, after which a choice of almost-complex structure is the same as a non-vanishing section of the rank-3 vector bundle $\Lambda^+$ of self-dual 2-forms on $W$.) 

Assuming an almost-complex structure exists, the obstructions to homotopy between two such sections lie in the cohomology groups $H^i(W; \pi_i(S^2))$. For $i = 2$ this corresponds to the difference between the induced \spinc structures. If $W$ is closed and spin, then there is an obstruction to homotopy lying in $H^4(W; \zee/2)$ (c.f. Kirby-Melvin-Teichner \cite{KMT}). In the cases of interest to us, we have the following:

\begin{prop}\label{jprop} Let $W$ be a 4-dimensional almost-complex manifold, and $Y$ a (nonempty) component of $\partial W$, such that $H^3(W, Y; \zee) = 0$. Then an almost-complex structure $J$ on $W$ is determined up to homotopy by its induced \spinc structure and the homotopy class of the oriented plane field $TY\cap J(TY) \in \P(Y)$.
\end{prop}

\begin{proof} If the \spinc structures $\s_{J_i}$ induced by $J_0$ and $J_1$ agree, then since $H^4(W) = 0$ and $\pi_3(S^2) = \zee$, the only obstruction to homotopy lies in $H^3(W; \zee)$. Under our hypotheses the restriction $H^3(W) \to H^3(Y)$ is injective, so it suffices to consider the restriction of $J_i$ to $T_Y W$. It is not hard to see that almost-complex structures on $T_YW$ are homotopic if and only if the corresponding plane fields $TY\cap J_i( TY)$ are homotopic in $TY$. 
\end{proof}

\begin{remark} It is also the case that for a general cobordism $W:Y_1\to Y_2$ between connected oriented 3-manifolds, if $J_1$ and $J_2$ are almost-complex structures inducing the same $\spinc$ structure on $W$ and also the same $2$-plane field on $Y_1$ then they induce the same plane field on $Y_2$.
\end{remark}

With this result, we can give a more detailed argument for the absolute grading of the contact invariant along the lines of the one outlined by Ozsv\'ath and Szab\'o in \cite{Contact}. For simplicity we work with coefficients in $\F$, though the discussion applies with minimal modification to determine the grading of the contact submodule in the case of twisted coefficients. 

Given an oriented 3-manifold $Y$ with co-oriented (and hence oriented) contact structure $\xi$, Giroux's theorem \cite{Giroux} implies that we can find an open book decomposition for $Y$ that supports $\xi$, has connected binding $K$, and has pages $S$ of genus $g>1$. Then Lemma \ref{fulltwistlemma} implies that the contact invariant of $\xi$ is equal to the image under $F^+_W$ of the nonzero element of $HF^+(-Y_0, \s_{1-g}) \cong \F$, where $W: -Y_0\to -Y$ is the surgery cobordism (in fact, in this untwisted situation Lemma \ref{fulltwistlemma} is nothing but \cite[Proposition 3.1]{Contact}).


Now construct a Lefschetz fibration $X$ with oriented boundary $-Y_0$ and whose singular fibration extends the surface bundle structure on the boundary. Removing a 4-ball from $X$, Ozsv\'ath and Szab\'o show in \cite[Theorem 5.3]{OSSymp} that there is a unique \spinc structure $\kk$ on $X$ having $\langle c_1(\kk), \widehat{S}\rangle = 2-2g$ and inducing a nontrivial map (in fact an isomorphism) $HF^+(Y_0, \s_{1-g})\to HF^+_0(S^3)$, where the latter refers to the absolute $\cue$-grading. Dualizing, it follows that the contact invariant $c(\xi)$ is equal to the ``mixed invariant'' $F_{X\cup W}^{mix}(\Theta^-)$ (c.f. Plamenevskaya \cite[Lemma 1]{OlgaDistinct} for more details). We are interested in the grading of this element. 

First consider the Lefschetz fibration $X$: it can be given a handle structure with a single 0-handle, $2g$ 1-handles, various 2-handles corresponding to the 2-handle of $\hat\widehat{S}{S}$ and a factorization of the monodromy of the fibration of $-Y_0$ into a product of right-handed Dehn twists, and no other handles. It follows easily that $H^3(X- B^4, S^3; \zee) = 0$, so that Proposition \ref{jprop} applies. Now, $X$ admits a canonical symplectic structure \cite[Theorem 10.2.18]{GS} (c.f. \cite{Gompf2004}) to which is associated a natural homotopy class of compatible almost-complex structure, from which an element $J_0$ can be chosen so that the fibers of the Lefschetz fibration $X\to D^2$ are $J_0$-holomorphic. In particular the adjunction formula implies that $c_1(J_0)$ pairs with the fiber $\widehat{S}$ to give $2-2g$. Let $\J_{\widehat{S}}(X)$ denote the set of homotopy classes of almost-complex structures $J$ on $X$ such that $\s_J = \s_{J_0}$. 

\begin{lemma} The set $\J_{\widehat{S}}(X)$ admits a free transitive $\zee$ action, and the restriction $\J_{\widehat{S}}(X)\to \P(S^3)$ is an isomorphism of $\zee$-sets. 
\end{lemma}

\begin{proof} According to Proposition \ref{jprop}, the map $J\mapsto TS^3\cap J(TS^3)$ induces an injection $\J_{\widehat{S}}(X) \to \P(S^3)$. For the $\zee$ action, choose a nicely embedded path $[0,1]\to X$ connecting the two boundary components of $X$; identify its neighborhood with $[0,1]\times B^3$. Given an almost-complex structure $J$, regarded as a section of the unit sphere bundle of $\Lambda^+(X)$, we trivialize the latter over $[0,1]\times B^3$ and construct $J'$ such that $J' = J$ away from the arc, and over each $t\times B^3$ the two sections glue over $\partial B^3$ to give the Hopf map $S^3\to S^2$. This is easily seen to correspond to a generator of the $\zee$ action on $\P(S^3)$.
\end{proof}

Now $\P(S^3)$ is identified with the integers via the Hopf invariant \eqref{hopfinvt}, and the standard tight contact structure on $S^3$ has Hopf invariant 0. However, while the plane field $\xi_0 = TS^3 \cap J_0(TS^3)$ is isotopic to the standard contact structure on $S^3$, we are considering it as a plane field on $-S^3$. It is not hard to determine the effect of orientation reversal on the Hopf invariant: since the quantity $c_1^2 - 3\sigma - 2\chi$ vanishes for closed almost-complex 4-manifolds, we have
\begin{equation}\label{hopfconj}
h(\xi, Y) + h(\xi, -Y) = 1.
\end{equation}
Hence the Hopf invariant of $TS^3 \cap J_0(TS^3)$ is $+1$. Observe that $TY_0\cap J_0(TY_0)$ is the field of tangents to the fibration of $Y_0$, since the latter is fibered by $J_0$-holomorphic surfaces. 

\begin{cor}\label{fiberedgradingcor} Let $M$ be an oriented, fibered 3-manifold with oriented fiber $S$ having genus $g>1$. Write $[T\widehat{S}]$ for the plane field of oriented tangents to the fibers. Then $HF^+(M, \s_{1-g})$ is supported in absolute grading $[T\widehat{S}]-1$. 
\end{cor}

\begin{proof} Construct a Lefschetz fibration $X$ as above with boundary $\partial X = -M$ and the same oriented fiber $\widehat{S}$, equipped with its canonical \spinc structure. We've seen that grading level 1 in $HF^+(S^3)$ is related by a member of $\J_{\widehat{S}}(X-B^4)$ to the tangent field $[T\widehat{S}]$ (and no other plane field on $M$, by Proposition \ref{jprop}). Hence $HF^+_0(S^3)$ is related only to $[T\widehat{S}]-1$. Since the canonical $\spinc$ structure on $X$ induces an isomorphism $HF^+(M, \s_{1-g})\to HF^+_0(S^3)$, the result follows from Gripp and Huang's work.\end{proof}

\begin{cor}\label{contactgradingcor} If $Y$ is an oriented  3-manifold with oriented contact structure $\xi$, then the contact invariant $c(\xi)$ lies in $HF^+_{[\xi] -1}(-Y)$. 
\end{cor}

\begin{proof} We know $c(\xi)$ is the image of the generator of $HF^+(-Y_0, \s_{1-g})$ under the map induced by the surgery cobordism $W: -Y_0\to -Y$. It follows from Proposition \ref{jprop} that an almost-complex structure on $W$ is determined up to homotopy by the induced plane field on $Y_0$, since the restriction $H^i(W)\to H^i(Y_0)$ is injective for $i = 2, 3$. Recall that Eliashberg \cite{EliashbergFilling} has constructed a symplectic structure on $W$ for which the fibration on $Y_0$ is symplectic and such that the contact planes on $Y$ are symplectically positive. It follows easily that there is an almost-complex structure (compatible with the symplectic structure) on $W$ relating the contact field $[\xi]$ on $Y$ and the plane field $[T\widehat{S}]$ tangent to the fibers in $Y_0$. By the previous corollary, and appealing to Gripp and Huang's results again, we conclude that the generator of $HF^+(-Y_0, \s_{1-g})$ can map only into $HF^+_{[\xi]-1}(-Y)$. 
\end{proof}

Note that if $\widehat{S}$ is the oriented fiber of the fibration $\pi: Y_0\to S^1$, then the same oriented submanifold is the oriented fiber of $f\circ \pi: -Y_0\to S^1$, where $f$ is complex conjugation. Hence $HF^+(-Y_0, \s_{1-g})$ is indeed supported in degree $[T\widehat{S}]-1$. 

\begin{remark} If $c_1(\xi)$ is torsion, we can revert to the absolute $\cue$-grading of Floer homology. In this case, equation \eqref{hopfconj} with the above corollary says that the contact invariant lies in degree $-h(\xi, Y)$ of $HF^+(-Y)$. This was observed, in slightly different language, by Plamenevskaya. Indeed, \cite[Lemma 1]{OlgaDistinct}  implies the statement at hand by considering the grading shift of the mixed map  in Heegaard Floer homology.
\end{remark}

\subsection{Adding Twists to Open Books}
With these preliminaries in place, we return to the homomorphism 
\[
G: HF^+(-Y, \F)\to \ul{HF}^+(-Y_0(K), \F[C_n])
\]
appearing in the surgery triangle. Here, as before, our arguments are given for the case $g>1$ but carry over directly to the genus 1 case by replacing $\F$ by the Novikov field $\Lambda$. We leave the attendant adjustments of the following proofs to the reader.

 Comparing the definition of $G$, given in \eqref{f1def} below, with \eqref{cobordmap} defining cobordism maps in twisted Floer homology, we observe that $G$ is  the map on Floer homology  induced by a $2$-handle cobordism connecting $-Y$ to $-Y_0(K)$, followed by the change of coefficients homomorphism induced by the projection $\F[T,T^{-1}]\to \F[T,T^{-1}]/\langle T^n-1\rangle\cong \F[C_n]$.  Note that if $W: Y\to Y_0(K)$ is the standard $2$-handle cobordism such as the one considered above (which we typically view as a cobordism from $-Y_0(K)$ to $-Y$), then the cobordism under consideration here is $-W$. In particular, the almost-complex structure constructed by Eliashberg is not compatible with the orientation now under consideration. However, it is still the case that an almost-complex structure on $-W$ is determined by the induced plane field on $-Y_0(K)$. For the present purposes it is convenient to consider $-W$ as a cobordism from $Y_0(K)$ to $ Y$; recall that $Y$ carries a contact structure we denote by $\xi$, supported by an open book with connected binding $K$ and page genus $g$.

\begin{prop}\label{gradingprop} Consider the surgery cobordism $-W: Y_0(K) \to Y$, and let $J$ be an almost-complex structure on $-W$ such that the tangents to the fiber surfaces in $Y_0(K)$ are positively $J$-invariant. Let $[\eta]$ be the homotopy class of the plane field $TY \cap J TY$ on $Y$. Then $[\xi] = [\eta] + 2g-1$. 
\end{prop}

\begin{proof} Topologically, $-W$ is given by attaching a 2-handle $D^2\times D^2$ to $Y_0(K)\times 1 \subset Y_0(K)\times [0,1]$, where we glue $D^2\times S^1$ to the neighborhood of a section $\gamma$ of the fibration $Y_0(K)\to S^1$. 

On $Y_0(K)\times [0,1]$ we have an almost-complex structure $J_0$ defined by requiring that  the oriented tangents to the fiber be complex subspaces, and that the positive normal vector $R$ to the fibers has $J_0(R) = \partial_t$. 

We think of $D^2\times D^2$ as embedded in $\cee^2$ as follows (c.f. Gay-Stipsicz \cite{GayStipsicz} for a very similar construction, but with a key orientation difference). Let $f$ and $g$ be increasing smooth functions of $r$ satisfying:
\begin{itemize}
\item $f(r) = {1\over 2}$ for $r \leq {1\over 4}$, $f'(r) \geq 0$, and $f(r) = r$ for $r\geq {3\over 4}$.
\item $g(r) = r$ for $r\leq {1\over 4}$, $g'(r) \geq 0$, $g(r) < 1$ for $r< 1$, and $g(r) = 1$ for $r\geq 1$.
\end{itemize}
Then, using polar coordinates $(r_1,\theta_1, r_2,\theta_2)$ on $\cee^2$, we embed $S^1\times D^2$ in $\cee^2$ by
\[
F(\phi, r,\theta) = (f(r), \phi, g(r), \theta).
\]
Together with the copy of $D^2\times S^1$ given by $\{r_1\leq 1\mbox{ and } r_2  = 1\}$, the image of $F$ encloses a homeomorphic image of $D^2\times D^2$ in $\cee^2$ indicated schematically in Figure \ref{fig:handle} (our thanks to David Gay and Andr\'as Stipsicz for permission to use their figure).
\begin{figure}
\begin{centering}\includegraphics[width=3in]{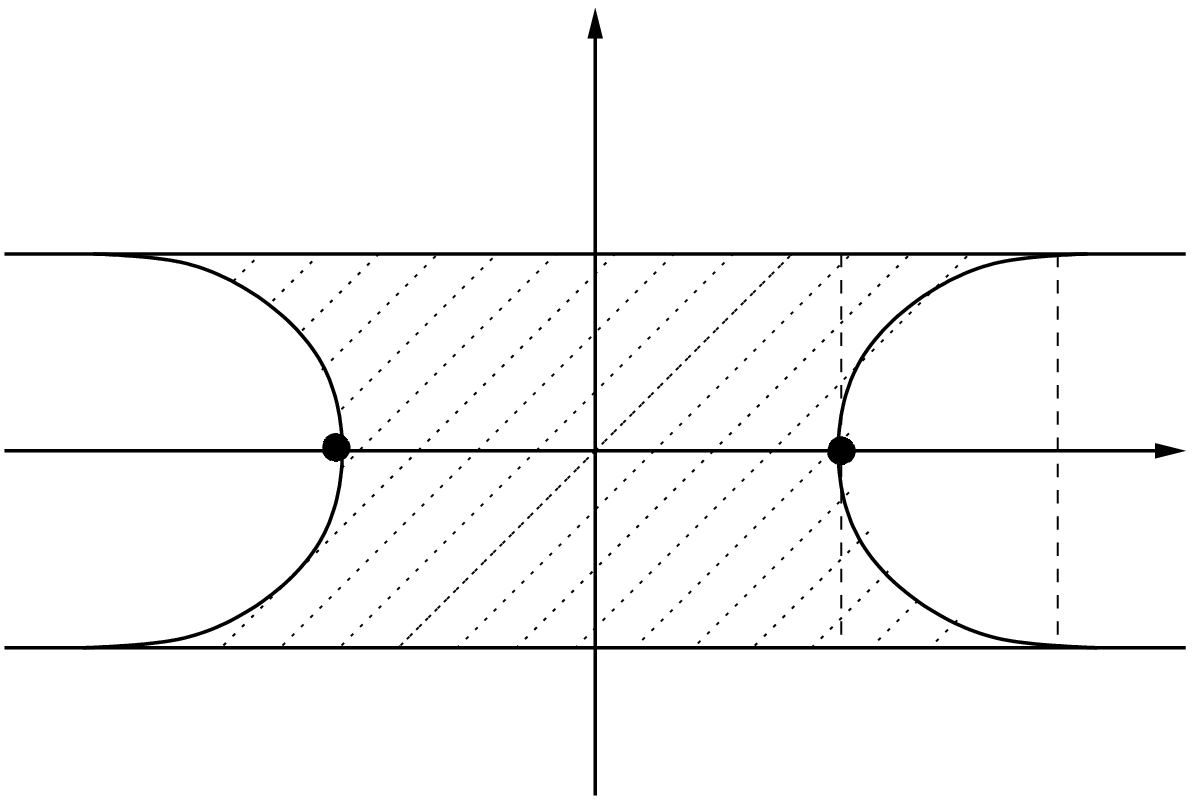}
\end{centering}
\put(-67,15){$1\over 4$}
\put(-27,15){$3\over 4$}
\put(2,61){$(r_1,\theta_1)$}
\put(-103,135){$(r_2,\theta_2)$}
\caption{\label{fig:handle}}
\end{figure}
With appropriate choices, we can embed $D^2\times S^1$ as a neighborhood of $\gamma\subset Y_0(K)\times 1$ in such a way that the almost-complex structure $J_0$ corresponds to the standard complex structure in $\cee^2$ near $\{r_1\leq 1 \mbox{ and } r_2 = 1\}$; in particular we can take the core circle $\gamma$ to correspond to $\{r_1 = 0 \mbox{ and } r_2 = 1\}$, while the intersection of the fibers with the neighborhood correspond to the disks $\theta_2 = const$. Then $J_0$ extends to the desired almost-complex structure over $-W = Y_0(K)\times [0,1] \cup D^2\times D^2$ and we must consider the complex tangencies on the new boundary.

Away from the attaching region, we still have the oriented tangents to the fiber. Near the surgery, an exercise with the embedding $F$ above shows that the complex tangencies on the surgery torus $S^1\times D^2$ (the neighborhood of the binding in $Y$) are given by 
\[
\eta = \ker(\beta = g' \,d\phi - f' d\theta)
\]
in coordinates $(\phi, r, \theta)$. Here the form $\beta$ determines the co-orientation, and hence orientation, of the planes. With our choices of $f$ and $g$, it follows that $\eta$ is a {\it negative} contact structure in the neighborhood of the binding; in particular the planes in $\eta$ make a quarter turn to the right as $r$ decreases from 1 to 0, and the oriented binding of the open book on $Y$ is negatively transverse to $\eta$.

To compare $\eta$ to the contact plane field $\xi$, recall that up to homotopy we can take $\xi$ to be tangent to the pages of the open book away from the binding. Near the binding, the planes in $\xi$ make a quarter turn to the left as $r$ decreases, so that the binding is positively transverse. Clearly, then, under the Pontryagin-Thom construction the framed links corresponding to $\xi$ and $\eta$ differ only in the neighborhood of the binding. To understand this difference, we use a special trivialization of $TY$. 

\begin{lemma}\label{trivlemma} Fix a page $S$ of the open book on $Y$. Then there is a trivialization $\tau = (e_1, e_2, e_3): Y \to TY$ with the following properties: 
\begin{enumerate}
\item The vectors $\{e_1,e_2\}$ provide an oriented trivialization of $\xi$ over the given page $S$.
\item The binding $K$ is tangent to $e_3$, and $(0,0,1)$ is a regular value of the maps $\xi^\tau$ and $\eta^\tau$. Moreover $(\eta^\tau)^{-1}(0,0,1)$ lies outside of the binding neighborhood $S^1\times D^2$. 
\end{enumerate}
\end{lemma}

Postponing the proof of the lemma for the moment, observe that the lemma implies that the difference between $[\xi]$ and $[\eta]$ can be calculated by taking the difference $[(\xi^\tau)^{-1}(0,0,1)] - [(\eta^\tau)^{-1}(0,0,1)]$, where the brackets indicate framed cobordism class. Clearly this difference is a framed copy of $K$, oriented as the boundary of $S$, with framing given by $\{e_1,e_2\}$. Since $K$ is null-homologous, we infer $[\xi]$ and $[\eta]$ differ only by the integer corresponding to this framing. To calculate the framing, we use $e_1$, say, to push off $K$ and calculate the resulting linking number. But since $\{e_1,e_2\}$ gives a trivialization of $\xi$ over $S$, this linking number is by definition the {\it self-linking number} $s\ell(K,S)$ of the transverse knot $K$. It is easy to see that the self-linking number of the binding of an open book is equal to $2g-1$ (for example recall that $s\ell(K,S)$ is minus the Euler number of $\xi\simeq TS$ over $S$, relative to a trivialization transverse to $\partial S$), so that $[\xi] - [\eta] = 2g-1$ as plane fields on $Y$.
\end{proof}

\begin{proof}[Proof of Lemma \ref{trivlemma}] Construct a trivialization $\tau$ over $S$ by taking $\{e_1,e_2\}$ to trivialize $\xi|_S$ and $e_3$ to be a Reeb vector field for a contact form $\alpha$ adapted to the open book.  Let $\sigma$ be a trivialization of $TY$ over all of $Y$; then the obstruction to homotopy between $\sigma$ and $\tau$ over $S$ lies in $H^1(S; \zee/2)$. This obstruction can be eliminated by twisting $\{e_1,e_2\}$ along circles generating the homology of $S$ while preserving the property (1) of $\tau$. Thus after suitable twisting, $\tau$ extends over the rest of $Y$, satisfying (1).

Property (2) is easily arranged by local perturbations, given our picture of $\xi$ and $\eta$ near $K$.
\end{proof}

Observe that when we consider $[\xi]$ and $[\eta]$ as plane fields on $-Y$ we have $[\eta] = [\xi] + 2g-1$ since reversing orientation negates framings.

\begin{cor}\label{uniqcor} The only elements of $HF^+(-Y)$ that map nontrivially to $\ul{HF}^+(-Y_0(K), \s_{1-g}; \F[C_n])$ under the surgery cobordism lie in degree $[\xi] + 2g-2$.
\end{cor}

\begin{proof} Since the target group is supported in degree $[T\Shat] -1$ this is another application of Proposition \ref{jprop}, this time with the orientation reversed on the cobordism. Indeed, the result just proved shows that $[T\Shat]$ is related through $-W$ only to $[\xi] +2g-1$.
\end{proof}

We phrase the following in terms of the grading of the contact invariant $c(\xi)$, since that seems more natural from the point of view of Heegaard Floer theory than the grading corresponding to the homotopy class $[\xi]$ itself. 

\begin{theorem}\label{boundthm}  Let $[\gamma]\in \P(-Y)$ denote the grading of the contact invariant $c(\xi)$, so that $[\gamma] = [\xi] -1$ as plane fields on $-Y$. Suppose that for some $n>0$ the restriction to the canonical grading $[T\widehat{S}] - 1$ of the map in the surgery triangle,
\[
 F: \ul{HF}^+_{[T\widehat{S}]-1}(-Y_0(K); \F[C_n])\to HF^+(-Y_{-1/n}(K); \F),\]
  vanishes. Then 
\[
\dim_\F(\coker(U)\cap HF^+_{[\gamma]+2g-1}(-Y; \F)) \geq n. \]
If $\s_\xi$ is self-conjugate, then 
\[
\dim_\F(\coker(U)\cap HF^+_{[\gamma]+2g-1}(-Y; \F)) \geq 2n.
\]
\end{theorem}

\begin{proof} The hypothesis holds if and only if the map $G_{-W}: HF^+(-Y, \F)\to \ul{HF}^+(-Y_0(K); \F[C_n])$ maps onto the summand in degree $[T\Shat]-1$.  Corollary \ref{uniqcor} shows that the only contribution to the image of the map $G_{-W}$ in that degree comes from its restriction  to  grading $[\xi]+2g-2 = [\gamma] + 2g-1$. By Lemma \ref{fulltwistlemma} and Corollary \ref{fiberedgradingcor} the Floer homology of $-Y_0(K)$ in degree $[T\Shat] -1$ has dimension $n$ over $\F$ and consists of elements in the cokernel of $U$. Since cobordism maps are $U$-equivariant, the first statement follows.

For the strengthened conclusion in the self-conjugate case, begin by noting that
by conjugation invariance, our hypotheses imply that the restriction of $F$ to the summand corresponding to $\s_{g-1}$ also vanishes. Moreover, from Lemma \ref{conjlemma}, if $M$ is fibered then $HF^+(M, \s_{g-1})$ is supported in degree $[T\Shat]^*-1$. Observe that if $J$ is the (unique) almost-complex structure on $-W: -Y\to -Y_0(K)$ inducing the tangent field $[T\Shat]$ on $-Y_0(K)$, and therefore restricting to $[\eta]= [\xi] +2g-1$ on $-Y$, then $-J$ induces $[T\Shat]^*$ and restricts to $[\eta]^*$. The same argument as above then shows $HF^+_{[\eta]^*-1}(-Y)$ surjects to $\ul{HF}^+_{[T\Shat]^*-1}(-Y_0(K)) \cong \F^n$, by the unique component of $G_{-W}$ connecting these two groups.  We claim that $[\eta]^* - 1 = [\eta]-1$, so that the single group $HF^+_{[\eta]-1}(-Y)$ must surject to $\F^{2n}$.

For this, observe that if a plane field $\zeta$ has the property that $c_1(\zeta)$ is a torsion class (which is certainly true if $\s_\zeta$ is self-conjugate), then the homotopy class of $\zeta$ is characterized by the \spinc structure $\s_\zeta$ together with the Hopf invariant of $\zeta$ \cite{Gompf}. Reversing the orientation of $\zeta$ has the effect of conjugating $\s_\zeta$, but preserves the Hopf invariant. (Put another way, conjugation of a torsion \spinc structure preserves the absolute $\cue$-grading.)

Now suppose $\s_\xi$ is self-conjugate, and observe that $\s_\eta = \s_\xi$ and has  Chern class which is 2-torsion. Then by the previous remark we see $[\eta] = [\eta]^*$, hence we must have a surjection
\[
HF^+_{[\eta]-1}(-Y)\to \ul{HF}^+(-Y_0(K), \s_{1-g}; \F[C_n]) \oplus \ul{HF}^+(-Y_0(K), \s_{g-1}; \F[C_n]).
\]
We conclude by noting that the target group above is $\F^{2n}$ and has trivial $U$-action, while $[\eta]-1 = [\xi]+2g-2 = [\gamma] + 2g-1$.\end{proof}

Let $M$ be any oriented 3-manifold, and $[\zeta]\in \P(M)$ a homotopy class of oriented plane field. We define
\[
\mathcal{K}(M, [\zeta]) = \dim_\F(\ker(U)\cap HF^{+,red}_{[\zeta]}(M;\F)),
\]
and dually
\[
\mathcal{K}^*(M, [\zeta]) = \dim_\F(\coker(U)\cap HF^{+,red}_{[\zeta]}(M;\F))
\]
(though by structure theorems for $HF^+$, the restriction to reduced Floer homology is redundant in the case of $\mathcal{K}^*$). 

\begin{cor}\label{mintwistcor} Let $\xi$ be a contact structure on a 3-manifold $Y$ supported by an open book with connected binding $K$. Let $\xi_n$ denote the contact structure obtained by adding $n$ right Dehn twists along the boundary of the page, and let $[\gamma]\in\P(-Y)$ denote the absolute grading of the contact invariant $c(\xi)$. Then
\[
n>\mathcal{K}^*(-Y,[\gamma]+ 2g-1) \quad \implies \quad \mbox{$\xi_n$ is tight}
\]

If $\s_\xi$ is self-conjugate, then
\[
2n>\mathcal{K}^*(-Y, [\gamma]+2g-1) \quad \implies \quad \mbox{$\xi_n$ is tight.}
\]
\end{cor}

\begin{proof} Combine Theorem \ref{boundthm} and Proposition \ref{tightprop}.
\end{proof}

\begin{cor}\label{twistboundcor} Let $K\subset Y$ be a fibered knot, and let $\tau_K$ be the fractional Dehn twist coefficient of the monodromy of $K$. Then we have an inequality
\[
-1-\mathcal{K}^*(-Y, [\gamma] + 2g-1) \leq \tau_K \leq 1+\mathcal{K}(-Y, [\gamma]),
\]
where $[\gamma]$ is the absolute grading of the contact invariant of the contact structure supported by the open book determined by $K$.

If the \spinc structure associated to this contact structure is self-conjugate, then in fact
\[
-1-\lfloor{\ts {1\over 2}}\mathcal{K}^*(-Y,[\gamma]+2g-1) \rfloor \leq \tau_K \leq  1+\lfloor{\ts {1\over 2}}\mathcal{K}(-Y, [\gamma])\rfloor.
\]
\end{cor}

\begin{proof}[Proof of Theorem \ref{thm3}] By definition, $\mathcal{K}^*(-Y,[\gamma]+2g-1)$ and $\mathcal{K}(-Y, [\gamma])$ are quotient- and sub-spaces of $HF^{red}_{2g-1-h(\xi)}(-Y,\s_\xi)$ and $ HF^{red}_{-h(\xi)}(-Y,\s_\xi)$, respectively, where we assume $c_1(\s_\xi)$ is torsion and use the fact that in this case $[\gamma]$ is identified with $-h(\xi)$. Hence the statement of the theorem follows immediately from the first inequality in Corollary \ref{twistboundcor}.
\end{proof}

\begin{proof}[Proof of Corollary \ref{twistboundcor}] For the first statement, observe that adding $1+\mathcal{K}^*(-Y, [\gamma] +2g-1)$ right twists to the monodromy of $K$ produces an open book supporting a tight contact structure by the previous corollary. The new monodromy has twist coefficient $\tau_K +1+ \mathcal{K}^*(-Y, [\gamma] +2g-1)$, which must be nonnegative since the supported contact structure is tight \cite{HKMVeer2}.

For the other inequality, observe that the fibered knot $K$ induces an open book on $-Y$ with oriented fiber $S$ and monodromy $\phi_K^{-1}$. Let $\bar{\xi}$ denote the associated positive contact structure on $-Y$ and write $[\bar{\gamma}]\in \P(Y)$ for the grading of the associated contact element. Then the result just obtained says 
\begin{equation}\label{oppineq}
-\tau_K = \tau_{\phi_K^{-1}} \geq -1- \mathcal{K}^*(Y, [\bar{\gamma}] + 2g-1) = -1-\mathcal{K}^*(Y, [\bar{\xi}] + 2g-2)
\end{equation}
by Corollary \ref{contactgradingcor}. To translate this back to an inequality involving the Floer homology of $-Y$ we make use of the following.

\begin{lemma}\label{duallemma} For any 3-manifold $M$ and oriented plane field $[\zeta]$ we have
\[
\mathcal{K}^*(M, [\zeta])=\mathcal{K}(-M, [\zeta]-2).
\]
\end{lemma}

Up until this point our notation for the action of $\zee$ on oriented plane fields---which depends implicitly on the ambient orientation---has not referred to this orientation. Here we must make this more precise, and we write $[\xi] \super{+}{Y} n$ to indicate the action of $n\in \zee$ on the homotopy class of plane field $[\xi]$. Observe that $[\xi] \super{+}{-Y} n = [\xi] \super{-}{Y} n$. With this in mind, we have
\begin{eqnarray}
\tau_K = -\tau_{\phi_K^{-1}} &\leq& \mathcal{K}^*(Y, [\bar{\xi}] \super{+}{Y} (2g-2))+1 \nonumber \\
&=& \mathcal{K}(-Y, ([\bar{\xi}] \super{+}{Y} (2g-2)) \super{-}{-Y} 2 )+1\label{line2}\\
&=& \mathcal{K}(-Y, [\bar{\xi}] \super{-}{-Y} 2g)+1\label{line3}\\
&=& \mathcal{K}(-Y, [\xi] \super{-}{-Y} 1)+1\label{line4}\\
&=& \mathcal{K}(-Y, [\gamma])+1.\label{line5}
\end{eqnarray}
Here the first line is from \eqref{oppineq}. Then \eqref{line2} is from the Lemma, and gives \eqref{line3} by the remark above. For \eqref{line4}, observe that as a plane field on $Y$, the (now negative) contact structure $\bar{\xi}$ stands in the same relationship to $\xi$ as did the plane field $[\eta]$ in Proposition \ref{gradingprop}: it is homotopic to $\xi$ away from the binding, but turns right instead of left as we approach the binding. Therefore the argument from before implies $[\xi] - [\bar{\xi}] = 2g-1$ as plane fields on $Y$, hence $[\bar{\xi}] = [\xi] \super{+}{-Y} (2g-1)$ as plane fields on $-Y$. Equation \eqref{line5} is Corollary \ref{contactgradingcor} again.

The bound on $\tau_K$ when $\s_\xi$ is self-conjugate follows similarly, using the stronger conclusion in Corollary \ref{mintwistcor} for this case.
\end{proof}

\begin{proof}[Proof of Lemma \ref{duallemma}] For the purposes of this proof, we write
\[
\mathcal{K}^{\pm}(Y, [\zeta]) = 1 + \dim_\F(\ker(U) \cap HF^{\pm, red}_{[\zeta]}(Y)),
\]
and $\mathcal{K}^{\pm, *}(Y,[\zeta])$ for the analogous quantity involving the cokernel of $U$. Observe that from the $U$-equivariant isomorphism $\tau: HF^{+, red}\to HF^{-, red}$ of degree $-1$, we obtain
\[
\mathcal{K}^-(Y,[\zeta])= \mathcal{K}^+(Y, [\zeta] + 1).
\]
Recall that there is a duality between the chain complex $CF^+(Y)$ and the {\it cochain} complex $CF_-(-Y)$. In terms of the absolute grading by plane fields we claim
\begin{equation}\label{duality}
HF^+_{[\zeta]}(Y) = HF_-^{[\zeta] - 3}(-Y).
\end{equation}
This is most easily seen in the case of a torsion \spinc structure, by reverting to the absolute $\cue$ grading using the Hopf invariant. Indeed, in terms of the absolute grading, Ozsv\'ath and Szab\'o \cite{AbsGrad} show that $HF^+_n(Y) = HF_-^{-n-2}(-Y)$. If $h([\zeta], Y) = n$, then from \eqref{hopfconj} we get $h([\zeta]- 3, -Y) = h([\zeta],{-Y}) - 3 = -h([\zeta],Y) - 2 = -n-2$.

In the general case recall that the duality between $HF^+(Y)$ and $HF_-(-Y)$ is obtained by considering a Heegaard diagram $(\Sigma, \aalpha, \bbeta, w)$ for $Y$ and the corresponding diagram $(\Sigma, \bbeta, \aalpha, w) = (-\Sigma, \aalpha, \bbeta, w)$ for $-Y$. One maps a generator $U^{-i}\cdot\x$ for $CF^+(Y)$ to the corresponding generator $(U^{i+1}\cdot \x)^*$ for $CF_-(-Y)$, where the star indicates the Kronecker dual. The key, then, is to understand the difference between $\tgr^{Y}(\x)$ and $\tgr^{-Y}(\x)$, where $\tgr$ denotes Gripp and Huang's plane field grading, and the superscript indicates the ambient manifold---though we consider the results as oriented plane fields independent of the orientation on $Y$. Now, the construction in \cite{GH} proceeds by using the intersection point $\x$, together with the basepoint $w$ of the Heegaard diagram, to give a standard modification of a gradient-like vector field for the Morse function on $Y$ and obtain a nonvanishing vector field that is well-defined up to homotopy. Then taking (oriented) orthogonal complements gives the desired plane field. Consider the effect of reversing the orientation of $Y$ while negating the gradient-like vector field: if the canonical modifications to the gradient field were insensitive to this change then we would infer $\tgr^Y(\x) = \tgr^{-Y}(\x)$ as oriented plane fields. 

However, this is not quite the case. Examining the constructions of \cite{GH} shows that the result of modifying the gradient vector field in the neighborhood of a flow line connecting critical points of index 1 and 2 is canonical, i.e., independent of orientation. However, the modification in the neighborhood of the flow line between the index 0 and 3 critical points determined by the basepoint $z$ {\it does} implicitly involve the orientation. This modification proceeds in two steps: first the gradient field is modified on a $B^3$ neighborhood of the flow line to a field that vanishes on an unknot contained in $B^3$. Then we add a circulatory vector field around this unknot, cut off outside of a small neighborhood. The choice of direction of this circulatory field is dictated by the ambient orientation. An analysis of the difference between the choices, which we omit here, leads to the conclusion $\tgr^{-Y}(\x) = \tgr^Y(\x) \super{-}{-Y} 1$. (In fact, once the issue has been reduced to this local calculation, the general result follows from the case of a torsion \spinc structure together with the fact that the plane field grading lifts the absolute $\cue$ grading.)

Hence in the duality result we have 
\[
\tgr^{-Y}((U^{i+1}\x)^*) = \tgr^{-Y}(\x) \super{-}{-Y} (2i+2) = \tgr^{Y}(\x) \super{-}{-Y} 3 \super{+}{Y} 2i = \tgr^Y(U^{-i}\x) \super{-}{-Y} 3,
\]
 confirming \eqref{duality}.

To complete the proof of the Lemma, note that the dimension over $\F$ of the cokernel of the action of $U$ on $HF_-$ cohomology (in a given grading) is the same as the dimension of the kernel of the action of $U$ on $HF^-$ homology (in that grading). Therefore, \eqref{duality} implies
\begin{eqnarray*}
\mathcal{K}^{+,*}(M, [\zeta]) &=& \mathcal{K}^{-}(-M, [\zeta] -3) \\
&=& \mathcal{K}^+(-M, [\zeta] -2).
\end{eqnarray*}
\end{proof}

\section{Twisted Floer homology and the Surgery Exact Triangle}\label{trianglesec}
In this subsection we state and sketch a proof of a general  surgery exact triangle relating the (twisted) Floer homology of three $3$--manifolds obtained by Dehn filling a single manifold $M$ with torus boundary.   The discussion can be viewed as a synthesis and clarification of the literature.   

Before stating the theorem, we briefly recall that Heegaard Floer homology of a $3$--manifold $Y$ can be defined with coefficients in any $\F[H^1(Y;\Z)]$--module, by appealing to  standard constructions of homology with twisted coefficients (imported to the setting of Morse homology) and noting that the fundamental group of the configuration space of paths between the Heegaard tori is given by
$$ \pi_1(\mathcal{P}(\Ta,\Tb),\x)\cong \pi_2(\x,\x)\cong \Z\oplus H^1(Y;\Z).$$
The totally twisted Floer complex $\underline{CF}^\infty(Y)$ is thus freely generated over $\F[\Z\oplus H^1(Y;\Z)]\cong \F[U,U^{-1}]\otimes\F[ H^1(Y;\Z)] $ by $\Ta\cap\Tb$.  See \cite[Section 8]{HolDiskTwo}. The $\F[U,U^{-1}]$--module structure coming from the $\Z$ summand in $\pi_2(\x,\x)$ gives rise to a filtration (by complexes of $\F[U]$ submodules) of $\underline{CF}^\infty$ with sub-, quotient-, and subquotient complexes $\underline{CF}^-,\underline{CF}^+,\underline{\widehat{CF}}$.  We denote the collection of complexes by $\underline{CF}^\circ$.  Now   let $\Lambda$ denote the Novikov ring (which is a field, in this case)$$\Lambda:={\large{\{}} \underset{r\in \R}\Sigma  a_r \cdot t^r \ | \ a_r\in \F, \mbox{ and } \{a_r | a_r\ne 0, r< \lambda\} \text{ is finite for all } \lambda\in \R\},$$ with multiplication defined on monomials by $a_r t^r \cdot b_s t^s = a_r b_s t^{r+s}$ and extended linearly.   
A choice of two form $\omega\in \Omega^2(Y;\R)$ defines an $\F[H^1(Y)]$--module structure on $\Lambda$, where $\eta\in H^1(Y)\subset H^1(Y;\R)$ acts by:
$$\eta(\underset{r\in \R}\Sigma a_r t^r) := \underset{r\in \R}\Sigma a_r t^{(r+\int_Y \eta \wedge\omega)}.$$ Viewed  as an    $\F[H^1(Y)]$--module in this way, we denote the Novikov ring by $\Lambda_\omega$. See \cite{AiPeters} for a nice discussion. Similarly, given a closed curve $\gamma\subset Y$ we can define an $\F[H^1(Y)]$--module structure on the group algebra of the cyclic group $C_n=\Z/n\Z$ by:
$$  \eta( a_i \cdot \zeta^i):= a_i \cdot \zeta^{i+\eta([\gamma])}, \text{where} \ \zeta=e^{2\pi i /n}\in \F[C_n],$$
and $[\gamma]\in H_1(Y)$ is the homology class of the curve. When we view $\F[C_n]$ as a module in this way, we may refer to it as $\F[C_n]_\gamma$.  These definitions thus allow us to speak of Floer homology with coefficients in $\Lambda_\omega$ or $\F[C_n]_\gamma$:
$$ \underline{CF}^\circ(Y;\Lambda_\omega):=\underline{CF}^\circ(Y)\otimes_{\F[H^1]} \Lambda_\omega, \quad \underline{CF}^\circ(Y;\F[C_n]_\gamma):=\underline{CF}^\circ(Y)\otimes_{\F[H^1]} \F[C_n]_\gamma$$
Given $\omega\in \Omega^2(Y;\R)$, or a curve $\gamma\subset Y$, we can also amalgamate the  actions above to consider $\Lambda[C_n]$ as an $\F[H^1(Y)]$--module, where the action takes place on  $\Lambda$ and $C_n$ independently, as defined above. It will often be more convenient to use concrete models for these chain complexes, which will be described in the course of the proof of the following theorem.

\begin{theorem} \label{triangle} Let $M$ be an oriented $3$-manifold with oriented boundary $\partial M=T^2$, and let $\slope{0},\slope{1},\slope{2}\subset T^2$ be a triple of simple closed curves, whose algebraic intersection numbers satisfy (for some choice of orientations) $$\#\{\slope{0}\cap\slope{1}\}=-n, \quad \#\{\slope{1}\cap\slope{2}\}=\#\{\slope{2}\cap\slope{0}\}=-1,$$
where $n>0$. Then for any $2$-form $\omega$  which vanishes on $\partial M$, and for $\RR=\F$ or  $\Lambda_\omega$, there is a long exact sequence 
\[\begin{tikzpicture}[>=latex] 
\matrix (m) [matrix of math nodes, row sep=2em,column sep=2em]
{ HF^+(M_0; \RR)& & HF^+(M_1; \RR),\\
&  \underline{HF}^+(M_2; \RR[C_n]) & \\};

\path[->,font=\scriptsize]
(m-1-1) edge[->] node[above]  {} (m-1-3)
(m-1-3) edge[->] node[below] {$G$} (m-2-2)
(m-2-2) edge[->] node[below] {$F$} (m-1-1); 
\end{tikzpicture}\]
where $M_{j}$ is the $3$-manifold obtained by Dehn filling $M$ with  slope $\slope{j}$. The module structure on $\RR[C_n]$ is defined by the curve $\slope{2}^*$ obtained as the core of the filling torus and, for $n=1$, is isomorphic to $\RR$.  In each case $\Lambda_\omega$ should be interpreted as the module associated with the extension of $\omega$ by zero to a $2$--form over the filling solid torus.

The maps $G$ and $F$ are related to the maps on twisted Floer homology groups  induced by the canonical $2$-handle cobordisms between the filled $3$-manifolds, and are defined by chain maps in Equations \eqref{f1def} and \eqref{f2def} below.
\end{theorem}

\begin{remark} The assumption on the intersection numbers is equivalent to the condition that the slopes satisfy $$[\slope{0}+\slope{1}+n\slope{2}]=0\in H_1(\partial M).$$
\end{remark}

\begin{remark} The theorem also holds with $\Z$ replacing $\F$ throughout, and for the other versions of Floer homology provided that we complete the coefficients with respect to  $U$ in the case of minus and infinity.
\end{remark}

Before proving the theorem, we discuss a collection of closely related results  in the literature.  To begin, the theorem with $\RR=\Z$ was first proved in \cite[Theorem 9.14]{HolDiskTwo}, in the (not-so) special case that $M=Y\setminus \mathrm{nbd}(K)$ is the complement of a null-homologous knot in a homology sphere,  $\slope{1}$ its meridian, and $\slope{2}$ its Seifert longitude. This yields an exact triangle for the Floer homologies of the triple $Y_{1/n}(K),Y, Y_{0}(K)$  with twisting on the zero surgery term.  In the same paper, the case where $\slope{2}$ is the meridian of a null-homologous knot and $\slope{0}$ its Seifert longitude was also addressed, yielding an exact triangle for $Y_0(K),Y_n(K),Y $, with trivially twisted coefficients for the Floer homology of $Y$ (groups which are isomorphic to the direct sum of $n$ copies of the untwisted Floer homology).  In both cases, the proof relied on an adaptation of Floer's argument for an exact triangle in instanton homology \cite{Floer}. In particular, the long exact sequences came from short exact sequences on the chain level. This left  the geometric meaning of the connecting homomorphisms  unclear.  This was remedied for the fractional surgery exact triangle in \cite[Section 3.1]{AbsGrad}, where the  maps  starting and terminating on the twisted term were  interpreted in terms of holomorphic triangle counts in a cover of the symmetric product of a Heegaard diagram (the third map, too, was identified with triangle counts, but this fact was already explicit in \cite{HolDiskTwo}).  

In \cite[Theorem 4.5]{Branched}, the exact  sequence with $\RR=\Z/2\Z$ and $n=1$ was reproved in such a way to put all of the maps  on equal footing.  In particular, each map was defined using the same  holomorphic triangle counts involved in the definition of the theory's $2$--handle cobordism maps;  indeed, in the case $n=1$,  consecutive pairs of $3$--manifolds  in the triangle are manifestly cobordant through a single $2$--handle attachment. Each map in the exact sequence is a sum of the  maps on Floer homology induced by this cobordism.  The key for this approach was a break from Floer's proof of the exact triangle, and the implementation of an ``exact  triangle detection lemma"  \cite[Lemma 4.2]{Branched} (see proof below a statement).   In \cite[Theorem 3.1]{AiPeters}, this approach was extended to the case $n=1$ and $\RR=\Lambda_\omega$.  

In \cite[Theorem 3.1]{IntegerSurgeries}, the case with $\RR=\Z$, $M=Y\setminus \mathrm{nbd}(K)$ a null-homologous knot complement, $\slope{2}$ its meridian, and $\slope{0}$ an $m$--framed longitude was treated.  This yields an exact triangle between the Floer homology of $Y_{m}(K)$, $Y_{m+n}(K)$, and $n$ copies of that of $Y$.  There, the treatment was again via the exact triangle detection lemma, but the discussion left ambiguous the precise definition of certain maps relevant in the application of the lemma.    A theorem equivalent to Theorem \ref{triangle} in the case $\RR=\Z$ was stated  for the completed minus version of Floer homology \cite[Proposition 9.5] {ManolescuOzsvath}.    We turn to the proof.

\begin{proof} Since many elements of the proof appear in the literature,  we will outsource various details to specific references, and focus on issues that are either absent or  ambiguous elsewhere.  An easily stated version of the exact triangle detection lemma says that if $A_i$ are chain complexes, $f_i$ chain maps, and $h_i$ chain homotopies arranged as:
\[\begin{tikzpicture}[>=latex] 
\matrix (m) [matrix of math nodes, row sep=2em,column sep=2em]
{ A_0 & & A_1,\\
&  A_2 & \\};
\path[->,font=\scriptsize]
(m-1-1) edge[->] node[below]  {$f_0$} (m-1-3)
(m-1-3) edge[->] node[left] {$f_1$} (m-2-2)
(m-2-2) edge[->] node[right] {$f_2$} (m-1-1)
(m-1-1) edge[->,bend right=55] node[below] {$h_0$} (m-2-2)
(m-1-3) edge[->,bend right=35]  node[above] {$h_1$} (m-1-1)
(m-2-2) edge[->,bend right=55]  node[below] {$h_2$} (m-1-3);
\end{tikzpicture}\]
 which satisfy, for each $i\in \{0,1,2\}$ (regarded cyclically):
\begin{enumerate}
\item  (Null-homotopy) $f_{i+1}\circ f_{i}=\partial_{i+2}\circ h_{i} + h_{i}\circ \partial_{i}$
\item (Quasi-isomorphism) $f_{i+2}\circ h_{i} +  h_{i+1}\circ f_{i}$  induces an isomorphism on homology, 
\end{enumerate}
then the maps induced by  $f_i$  form a long exact sequence on homology.    Chain complexes, maps, and homotopies satisfying these assumptions, and which induce the desired exact triangle are produced from a specific Heegaard quadruple diagram:
$$(\Sigma;\{ {\boldsymbol \alpha}, \GGG{0},\GGG{1},\GGG{2}\}; \{w, p\}),$$
where

\begin{enumerate}
\item $\Sigma$ is a closed oriented surface of genus $g$,

\item $\GGG{i}=\{\gamma_1^i,\ldots,\gamma_g^i\}$, $i=0,1, 2,$ are  $g$-tuples of simple closed curves in $\Sigma$, arranged  so that the first $g-1$ curves are all small Hamiltonian translates of each other, and so that $\gamma_g^i$ live in a torus connect summand of $\Sigma$ and intersect  minimally in the same way as the filling slopes. 
\item ${\boldsymbol    \alpha}=\{\alpha_1,\ldots,\alpha_g\}$ is a $g$--tuple of homologically  independent, pairwise disjoint, simple closed curves in $\Sigma$, transverse to the union of $\GGG{j}$.

\item For each $i=0,1,2,$ the Heegaard diagram $(\Sigma,{\boldsymbol    \alpha}, \GGG{i})$ specifies $M_i$.

\item $w$ is a basepoint in the complement of all curves, and  $p$ is a basepoint in $\gamma_g^2$.

\item The diagram is admissible, in the sense that any multi-periodic domain satisfying  $n_w(\mathcal{P})=0$ either has at least one negative coefficient or satisfies $\omega([\mathcal{P}])> 0$, where  $\omega$ is the perturbation $2$-form.
\end{enumerate}

In terms of this diagram we define chain groups 
$$ A_i = \left\{\begin{array}{ll} 
		 \underset{\x \in \mathbb{T}_{\boldsymbol\alpha^{\phantom{i}}}\!\!\!\cap \Tg{i}}\bigoplus \RR^+\cdot\x & i=0,1,\\
		 \underset{\x \in \mathbb{T}_{\boldsymbol\alpha^{\phantom{i}}}\!\!\!\cap \Tg{i}}\bigoplus \RR^+[C_n]\cdot\x & i=2.\\
\end{array}\right.
$$
Here $\RR^+$ is the $\RR[U]$ module $\RR[U,U^{-1}]/U\cdot\RR[U]$.  The boundary operators $\partial_i:A_i\rightarrow A_i$ are given by:
$$\partial_i(U^{-j}\cdot \x)= \sum_{\stackrel{\scriptstyle\phi\in \pi_2(\x,\y)}{{\mu(\phi)=1}}}\#{\widehat{\mathcal{M}}(\phi)}\cdot t^{\omega(\phi)}\cdot U^{n_w(\phi)-j}\cdot\y, \quad \text{for } i=0,1,$$
and 
$$\partial_2(\zeta^k \cdot U^{-j}\cdot \x)= \sum_{\stackrel{\scriptstyle\phi\in \pi_2(\x,\y)}{{\mu(\phi)=1}}}\#\widehat{\mathcal{M}}(\phi)\cdot t^{\omega(\phi)}\cdot \zeta^{k+n_p(\partial \phi)}\cdot U^{n_w(\phi)-j}\cdot\y.$$
\noindent In the case that $\RR=\F$, the coefficients can be obtained from the above formulae by setting $t=1$.   Here, the  exponent of $t$ is given by the evaluation of $\omega$ on the two-chain in $Y$ arising from the domain of the Whitney disk (viewed as a two-chain on $\Sigma$), together with two-chains that cone off its boundary with gradient flowlines to the index one and two critical points of a Morse function on $Y$ specifying the Heegaard diagram.   In the $\RR[C_n]$ twisted case, we further multiply by the $n$-th root of unity, raised to the algebraic number of times the boundary of the domain crosses the $p$ basepoint.  Tracing through the definitions, one see that these complexes compute the  Heegaard Floer  groups in the theorem:
$$ H_*(A_i,\partial_i) \cong HF^+(M_i,\RR), \quad i=0,1,$$
$$ H_*(A_2,\partial_2) \cong HF^+(M_2,\RR[C_n]).$$
(For the $\RR[C_n]$ twisted complexes, the key point is that  the multiplicity $n_p(\partial \phi)$  equals  the intersection number of $\partial\mathcal{D}(\phi)$ with a curve that intersects $\gamma_g^2$ exactly once and no other curves; such a curve is isotopic to the core of the filling solid torus.) 

The hypotheses required by the triangle detection lemma will follow from the $A_\infty$ structure present in the Fukaya category of the symmetric product of $\Sigma$, together with the standard nature of the tori $\Tg{j}$ coming from the collections $\GGG{j}$, $j=0,1,2$.  Most of the gross features of the argument appear in the aforementioned references (see especially \cite[Proof of Theorem 3.1]{IntegerSurgeries} and \cite{OzsvathLectures}).  The new technical challenges reside primarily in understanding exactly how the twisting should be incorporated in the definition of the chain maps $f_i$ and  homotopies $h_i$, and how these definitions affect  algebraic and geometric  aspects of the argument.   Since these details are particularly relevant to the proof of our main theorem, we will try to provide a thorough treatment.

To begin, we must consider the (twisted) completed minus Floer complexes $\CFmc(\Tg{i},\Tg{i+1})$ for $i=0,1,2$.  This notation seems to be dominant in the literature, but we should note that it differs from  \cite[Section 2.5]{OzsvathLectures} where the complexes are denoted $CF^{--}$.  In each case, the  complex is freely generated  by $\x \in \Tg{i}\cap \Tg{i+1}$. For $i=0$ the ground ring is $\RR[[U]]$ and for $i=1,2$ we use   $\RR[[U]][C_n]$.  The boundary operators are defined as above, with the cases $i=1,2$ accounting for the multiplicity of domains of Whitney disks at $p\in \gamma_g^2$. The reason to consider power series in $U$ is that there may be infinitely many homotopy classes of Whitney polygons defined by the Heegaard diagram which admit holomorphic  representatives.  The admissibility conditions placed on our  diagram  ensure, however, that there are only finitely many such homotopy classes with fixed $n_w(\psi)$.  It follows that the polygon counts can be used to define maps between the completed minus (or infinity) groups.   

Observe that the $3$--manifold specified by $(\Sigma,\Tg{i},\Tg{i+1})$ is homeomorphic to $\#^{g-1}S^1\times S^2$ when $i=1,2$, while $(\Sigma,\Tg{0},\Tg{1})$ specifies the connected sum $L(n,1)\#^{g-1}S^1\times S^2$.  Their Floer homologies are given as follows:
$$ \HFmc(\Tg{0},\Tg{1})\cong \RR^n\otimes \Lambda^*(\RR^{g-1})\otimes \RR[[U]]$$
$$\HFmct(\Tg{i},\Tg{i+1})\cong \Lambda^*(\RR^{g-1})\otimes \RR[[U]][C_n], \ \ \mathrm{for}\ i=1,2.$$
One can compute this directly, or apply the K{\"u}nneth theorem for the (completed) Floer homology of a connected sum of $3$--manifolds.  For $i=1,2,$ the highest graded summand of the Floer group is rank one over $\RR[C_n]$, and we denote a generator by $\Theta_{i,i+1}$. For the $i=0$ case, the  $n$ summands  correspond to the  $n$ different $\spinc$--structures on $L(n,1)$.  Picking a particular $\spinc$--structure we  obtain a  top-dimensional generator for its summand, which we denote $\Theta_{0,1}$.  Note that this involves a choice of $\spinc$--structure on $L(n,1)$;  for each such choice we would obtain a (presumably different, in general) exact triangle.  
 
We can now define the chain maps and homotopies which serve as input for the exact triangle detection lemma.   The chain maps are given as follows:
\begin{equation}\label{f0def}  f_0(U^{-j}\x):=  \sum_{\stackrel{\scriptstyle\psi\in \pi_2(\x,\Theta_{0,1},\y)}{{\mu(\psi)=0}}}\#{{\mathcal{M}}(\psi)}\cdot t^{\omega(\psi)}\cdot U^{n_w(\psi)-j}\cdot\y
\end{equation}

\begin{equation}\label{f1def} f_1(U^{-j}\x):=  \sum_{\stackrel{\scriptstyle\psi\in \pi_2(\x,\Theta_{1,2},\y)}{{\mu(\psi)=0}}}\#{{\mathcal{M}}(\psi)}\cdot t^{\omega(\psi)}\cdot \zeta^{n_p(\partial \psi)}\cdot U^{n_w(\psi)-j}\cdot\y 
  \end{equation}
\begin{equation}\label{f2def} f_2(U^{-j}\zeta^k\x):= \sum_{\stackrel{\scriptstyle\psi\in \pi_2(\x,\Theta_{2,0},\y)}{\stackrel{{\mu(\psi)=0}}{{{n_p(\partial\psi)=-k \ \text{mod}\  n}}}}}\#{\mathcal{M}}(\psi)\cdot t^{\omega(\psi)}\cdot U^{n_w(\psi)-j}\cdot\y
  \end{equation}

Note that the map $f_i$ is defined by counting holomorphic triangles with boundary mapping to $\Ta,\Tg{i},\Tg{i+1}$, and with the vertex that maps into $\Tg{i}\cap\Tg{i+1}$ sent to our distinguished generator $\Theta_{i,i+1}$.  In each case the Novikov ring enters as with the definition of the boundary operators: we simply measure the $\omega$ area of the coned-off domains of the  Whitney triangles.  The two chains arising from coning Whitney triangles are contained within the four-manifold $X_{\alpha,\gamma^{i},\gamma^{i+1}}$ specified by the Heegaard triple diagram via the construction of \cite[Section 8]{HolDisk}, and $\omega$ canonically extends to this four-manifold by our assumption that $\omega|_{\partial M}=0$.       The only difference between the maps, then, is how they incorporate the $C_n$ twisting: $f_0$ makes no use of it; $f_1$ uses it similarly to the boundary operator on $\underline{CF}^+(M_2)$, via the signed crossing number of the boundary of a triangle at the twisting basepoint $p$; $f_2$ incorporates the twisting by requiring triangles counted in the expansion of $f_2(\zeta^k\x)$ to have boundary which crosses $p$ {\em negative} $k$ times (modulo $n$).  Verification that these define chain maps is, as usual, a consequence of Gromov compactness together with a homotopy conservation principle; namely, that intersection numbers (in the case of the $U$ action and $C_n$ twisting) and $\omega$ areas (in the case of the Novikov twisting) are homotopy invariants of a class $\psi\in \pi_2(a,b,c)$ which are additive under decomposition of such a class into the juxtaposition of a triangle with a disk.

 Similarly, we define  homotopy operators using pseudo-holomorphic quadrilateral counts:
 \begin{equation}\label{h0def}  h_0(U^{-j}\x):=  \sum_{\stackrel{\scriptstyle\psi\in \pi_2(\x,\Theta_{0,1},\Theta_{1,2},\y)}{{\mu(\psi)=-1}}}\#{{\mathcal{M}}(\psi)}\cdot t^{\omega(\psi)}\cdot \zeta^{n_p(\partial \psi)}\cdot U^{n_w(\psi)-j}\cdot\y
\end{equation}

\begin{equation} \label{h1def}h_1(U^{-j}\x):=  \sum_{\stackrel{\scriptstyle\psi\in \pi_2(\x,\Theta_{1,2},\Theta_{2,0}\y)}{\stackrel{{\mu(\psi)=-1}}{{{n_p(\partial\psi)=0 \ \mathrm{mod} \ n}}}}}\#{{\mathcal{M}}(\psi)}\cdot t^{\omega(\psi)}\cdot  U^{n_w(\psi)-j}\cdot\y 
  \end{equation}
\begin{equation}\label{h2def} h_2(U^{-j}\zeta^k\x):= \sum_{\stackrel{\scriptstyle\psi\in \pi_2(\x,\Theta_{2,0},\Theta_{0,1}\y)}{\stackrel{{\mu(\psi)=-1}}{{{n_p(\partial\psi)=-k \ \text{mod} \ n}}}}}\#{{\mathcal{M}}(\psi)}\cdot t^{\omega(\psi)}\cdot U^{n_w(\psi)-j}\cdot\y
  \end{equation}
 
 If we consider one dimensional families of pseudo-holomorphic quadrilateral (arising from $\mu=0$ homotopy classes) then Gromov compactness, together with  additivity of  $\omega(\psi),n_w(\psi),n_p(\partial\psi)$  under juxtaposition, implies that  $h_i$ provides a homotopy between  $f_{i+1}\circ f_i$ and the operator:
 $$ f_{\alpha,i,i+2}(-\otimes f_{i,i+1,i+2}(\Theta_{i,i+1}\otimes\Theta_{i+1,i+2})),$$
 where $$f_{\alpha,i,i+2}: CF^+(M_i;\RR)\otimes \CFmc(\Tg{i},\Tg{i+2})\rightarrow CF^+(M_{i+2};\RR)$$
$$ f_{i,i+1,i+2}: \CFmc(\Tg{i},\Tg{i+1})\otimes \CFmc(\Tg{i+1},\Tg{i+2})\rightarrow \CFmc(\Tg{i},\Tg{i+2})$$
are chain maps defined by counting holomorphic triangles with appropriate boundary conditions (for these latter maps, we have suppressed notation  indicating which complexes are $C_n$ twisted, but remind the reader that the complexes for $M_j$ are twisted only when $j=2$, and the complexes for pairs $\Tg{i},\Tg{j}$ are twisted unless $\{i,j\}=\{0,1\}$).  For all of the maps, homotopies, etc. involved, the key idea to keep in mind is that if the map emanates from a $C_n$ twisted complex, then the holomorphic polygons counted must cross the twisting point $p$ a number of times equal to negative the exponent of the $\zeta$ power appearing in front of the intersection point.    Another notable feature is the requirement by $h_1$ that the $\Tg{2}$ boundary of the rectangles should cross $p$ zero times, modulo $n$.  This is actually a convention which is tied to our choice of $\spinc$-structure on $L(n,1)$ used to determine $\Theta_{0,1}$.  Choosing a different $\spinc$-structure would force us to require $n_p(\partial \psi)=m$ mod $n$ for some other value of $m$.

To verify that  $h_i$, so defined, is a null-homotopy for $f_{i+1}\circ f_i$, it suffices to show that $f_{i,i+1,i+2}(\Theta_{i,i+1}\otimes\Theta_{i+1,i+2})=0$.  This is essentially a local calculation in the torus summand of the Heegaard surface where the filling slopes  lie, together with a neck stretching argument and similar local considerations for the  torus summands where the other $\gamma$ curves lie.  See \cite[Proposition 2.10]{OzsvathLectures} for details on the argument, as applied to the hat theory, and \cite[Section 2.5]{OzsvathLectures} for its extension to plus.  For us, the only difference will be in the torus connect summand of the Heegaard surface where the filling slopes lie and the added bells and whistles that our twisting(s) incorporate.  The universal cover of this torus, together with the  lifts of the filling slopes,  is shown in Figure \ref{triangles} in the case where $n=3$.

\begin{figure}
\begin{center}
\def\svgwidth{4.6in}
 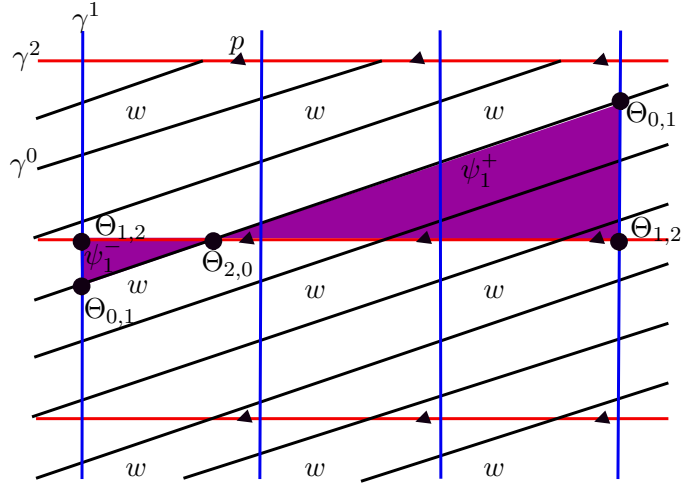
 \caption{\label{triangles} The universal cover of the torus summand of the Heegaard diagram where the filling slopes lie in the case $n=3$.  The black lines of slope $1/3$ represent lifts of  $\gamma_g^0$, and the blue vertical lines are lifts of $\gamma_g^1$.  The red horizontal lines are lifts of $\gamma_g^2$, and contain lifts of the basepoint $p$ which defines the $C_3$ twisting (we represent  lifts of $p$ by small black triangles).  Shown are the    triangles $\psi_1^\pm$ with  vertices on $\Theta_{i,i+1}$ and $\Theta_{i+1,i+2}$. They satisfy $n_w(\psi_1^\pm)=0,\ n_p(\partial \psi_1^\pm)= 0 \ \text{mod}\ 3$.}
\end{center}
\end{figure}

The key fact about this region for this part of the  argument  is that  triangles with two vertices on the $g$-th component of $\Theta_{i,i+1}$ and $\Theta_{i+1,i+2}$ and fixed values of $n_w(\psi),\omega(\psi)$ and $n_p(\partial\psi)\ \text{mod} \ n$ come in canceling pairs.  More precisely, for each $k>0$, there are exactly two triangles, $\psi_k^\pm$ with two vertices on the $g$-th component of $\Theta_{i,i+1}$ and $\Theta_{i+1,i+2}$, and these triangles satisfy $n_w(\psi_k^\pm)= n\cdot \frac{k(k-1)}{2}$ and $n_p(\partial \psi_k^\pm)=0 \ \text{mod} \ n$.
    That the triangles cancel comes from the facts that our base rings have characteristic two and that $\omega(\psi_k^+)=\omega(\psi_k^- )$ for all $k$. To see this latter fact, it suffices to observe that the cohomology class determined by $\omega$ on the four-manifold $X_{\gamma^{i},\gamma^{i+2},\gamma^{i+2}}$ is trivial, c.f. \cite[Theorem 3.1, last paragraph of proof]{AiPeters}.

We now turn to the quasi-isomorphism condition in the triangle detection lemma.      For this we consider an augmented Heegaard diagram which, in addition to the four sets of attaching curves previously mentioned, contains an additional $g$-tuple of curves $\GGGt{i}$ each of which arises via small Hamiltonian perturbation from a corresponding curve in $\GGG{i}$ (in particular, the 3-manifold specified by $(\Sigma,{\bf \alpha},\GGGt{i})$ is  homeomorphic to $M_i$).  We further require that each curve in $\GGGt{i}$ intersects the corresponding curve in $\GGG{i}$ in exactly two points. There are corresponding complexes, denoted $\tilde{A}_{i}$, and we  consider maps $g_i: A_i\rightarrow \tilde{A}_{i}$, defined by counting pseudo-holomorphic pentagons:  

\begin{equation} g_0(U^{-j}\x):=  \sum_{\stackrel{\scriptstyle\psi\in \pi_2(\x,\Theta_{0,1},\Theta_{1,2},\widetilde{\Theta}_{2,{0}},\y)}{\stackrel{{\mu(\psi)=-2}}{{{n_p(\partial\psi)=0 \ \mathrm{mod} \ n}}}}}\#{{\mathcal{M}}(\psi)}\cdot t^{\omega(\psi)}\cdot U^{n_w(\psi)-j}\cdot\y
\end{equation}

\begin{equation} g_1(U^{-j}\x):=  \sum_{\stackrel{\scriptstyle\psi\in \pi_2(\x,\Theta_{1,2},\Theta_{2,0},\widetilde{\Theta}_{0,{1}},\y)}{\stackrel{{\mu(\psi)=-2}}{{{n_p(\partial\psi)=0 \ \mathrm{mod} \ n}}}}}\#{{\mathcal{M}}(\psi)}\cdot t^{\omega(\psi)}\cdot  U^{n_w(\psi)-j}\cdot\y 
  \end{equation}
\begin{equation} g_2(U^{-j}\zeta^k\x):= \sum_{\stackrel{\scriptstyle\psi\in \pi_2(\x,\Theta_{2,0},\Theta_{0,1},\widetilde{\Theta}_{1,{2}},\y)}{\stackrel{{\mu(\psi)=-2}}{{{n_p(\partial\psi)=-k\ \text{mod}\ n}}}}}\#{{\mathcal{M}}(\psi)}\cdot t^{\omega(\psi)}\cdot \zeta^{n_{\tilde{p}}(\partial\psi)}\cdot U^{n_w(\psi)-j}\cdot\y,
  \end{equation}
 where $\widetilde{\Theta}_{i,i+1}$ is a top-dimensional generator for the complex associated to the Lagrangians coming from   $\GGG{i}$ and $\GGGt{i+1}$.  Note the appearance of $\tilde{p}$ in the last equation: this is a basepoint on $\tilde{\gamma}_g^2$ which is the image of $p$ under the Hamiltonian isotopy defining  $\gamma_g^2$.  
 
Gromov compactness for one dimensional families of pseudo-holomorphic pentagons, applied in this context, implies that such a family will have ten types of ends. Five arise from the  non-compactness of the domain coming from the vertices (boundary punctures) of the pentagon.   Using the fact that the $\Theta$ intersection points are cycles rules out three of the these ends, and the remaining two give rise to terms of the form $g_i\circ\partial_i+ \tilde{\partial}_{{i}}\circ g_i$.   The  other five ends correspond to  ends of the moduli space of conformal structures on a pentagon, over which the moduli spaces $\mathcal{M}(\psi)$ fiber. Each of these comes from a  conformal degeneration of a pentagon into a rectangle and triangle joined at a vertex.  Of these, two give rise to the  terms in the sum of compositions $\tilde{f}_{i+2}\circ h_{i} +  \tilde{h}_{i+1}\circ f_{i}$, where $\tilde{f}_{i+2}$ and $\tilde{h}_{i+1}$ are defined exactly as in Equations \eqref{f0def}--\eqref{h2def}, but with the $\GGGt{i}$ curves used in place of $\GGG{i}$ in the range of the map.  Two  of the remaining ends  involve triangles which contribute to the maps   $f_{i,i+1,i+2}(\Theta_{i,i+1}\otimes \Theta_{i+1,i+2})$ and ${f}_{i+1,i+2,\tilde{i}}(\Theta_{i+1,i+2}\otimes \widetilde{\Theta}_{i+2,i})$, which were previously shown to be zero.  The remaining ends  contribute to the map
 \begin{equation}\label{quasi} q_i(-):=f_{\alpha,i,\tilde{i}}(-\otimes h_{i,i+1,i+2,\tilde{i}}(\Theta_{i,i+1}\otimes \Theta_{i+1,i+2}\otimes \widetilde{\Theta}_{i+2,{i}}))\end{equation}
where $h_{i,i+1,i+2,\tilde{i}}$ is an operator defined by counting holomorphic quadrilaterals.   
Thus the pentagon operators provide a chain homotopy between $\tilde{f}_{i+2}\circ h_{i} +  \tilde{h}_{i+1}\circ f_{i}$ and $q_i$.   We claim that $q_i$ induces an isomorphism on homology (in fact, it is an isomorphism of chain complexes, but we will not need this).  Granting this, we have essentially proved the theorem.  The one caveat is that $q$ is not a map from $A_i$ to itself, but to a (quasi-)isomorphic complex $\tilde{A}_i$.   The easiest way around this technicality is to tweak the detection lemma to address a  family of chain complexes which have three-periodic homology.  This  is the route taken by \cite{Branched} and subsequent incarnations. We follow suit, so that our $f,h,$ and $g$ maps  increase the index (by $1$,$2$, and $3$, respectively) in the family of complexes $\{A_i\}_{i\in \Z}$ which we will show have three-periodic homology via   $q_i: A_i\rightarrow A_{i+3}$, with $A_{i+3}:=\tilde{A}_i$.

  Working with this setup, it only remains  to show that $q_i$ induces an isomorphism on homology.   When $i\ne 2$ mod $3$ it will suffice to show that 
$$ \widehat{h}_{i,i+1,i+2,\tilde{i}}(\Theta_{i,i+1}\otimes \Theta_{i+1,i+2}\otimes\Theta_{i+2,\tilde{i}})= t^{\lambda}\cdot\Theta_{i,\tilde{i}},$$
for some $\lambda$ (since $t^\lambda$ is a unit in $\Lambda$), and that 
$$ \widehat{f}_{\alpha,i,\tilde{i}}(-\otimes \Theta_{i,\tilde{i}})$$ induces an isomorphism on homology, where in both cases the ``hat" refers to the induced map on the corresponding hat Floer complex (that verification of isomorphism for the hat complex implies it for the plus complex is a consequence of \cite[Exercise 1.4]{OzsvathLectures}).
Verifying the former is essentially the same argument found in \cite[Discussion surrounding Equation 15, Figures 8 and 9]{OzsvathLectures},  the only real difference 
being the local calculation in the torus region where the filling slopes lie and the implicit  Novikov twisting.  Indeed, we obtain the factor of $t^\lambda$ in front of $\Theta_{i,\tilde{i}}$, where $\lambda$ is the $\omega$ area of the coned off domain of the unique pseudo-holomorphic quadrilateral with $n_p(\partial \psi)=0$ modulo $n$ and $n_w(\psi)=0$.

  For the latter, when $i\ne 2$ mod $3$, one can easily show that  that $f_{\alpha,i,\tilde{i}}(-\otimes \Theta_{i,\tilde{i}})$ is an isomorphism by arguing that it agrees, up to higher order terms with respect to the  area filtration, with the ``closest point" map $\iota$ discussed in \cite[Proof of Lemma 2.17]{OzsvathLectures}.  When Novikov coefficients are used, one needs to be careful if a non-admissible Heegaard diagram is employed:  in that case one cannot find an area form which vanishes on all periodic domains, hence the area filtration is not well-defined.  The argument still works, however, if the Heegaard diagram is  admissible in the weaker sense that $\omega$ evaluates positively on  positive multi-periodic domain. For then one can filter the complex using a combination of  area and the natural filtration of $\Lambda_\omega$ by powers of $t$.

\begin{figure}[b]
\begin{center}
\def\svgwidth{6.6in}
 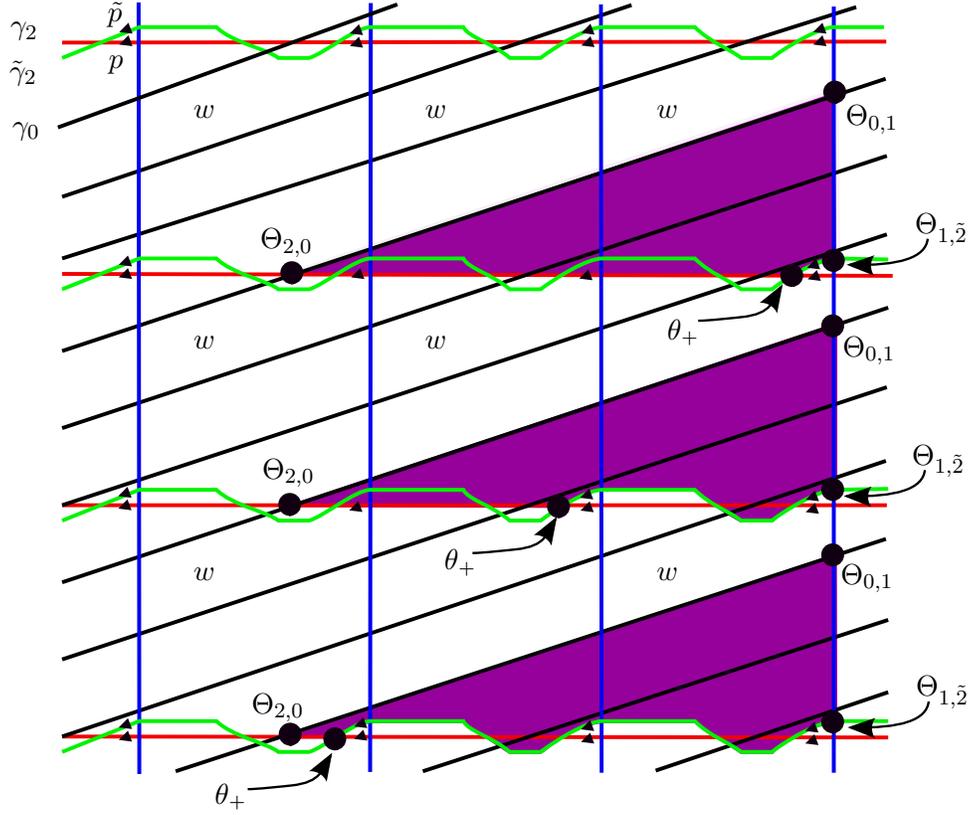
 \caption{\label{quads} The figure shows the domains of $n=3$ holomorphic quadrilaterals embedded in the  universal cover of the torus summand where the filling slopes lie.   These  account for the terms in the sum  (\ref{eq:quad}). Each can be viewed as a slight perturbation of the triangle $\psi_1^+$ from Figure \ref{triangles}, and they differ only in which lift of the $g$-th component of $\theta_+$ the boundary of the quadrilateral ``jumps" from $\tilde{\gamma}^2$ to $\gamma^2$.  This difference affects the values of $n_p(\psi)$ and $n_{\tilde{p}}(\psi)$, giving rise to the different terms in Equation (\ref{eq:quad}).  The top, middle, and bottom quadrilaterals give rise to the terms $\zeta^{2}\theta_+ \tilde{\zeta}^{1}, \ \zeta^{1}\theta_+ \tilde{\zeta}^{2},$ and $\zeta^{0}\theta_+ \tilde{\zeta}^{3}$, respectively.}
\end{center}
\end{figure}

  The cases with $i=2$ mod $3$ are somewhat different than the other two.  Here the chain maps considered in Equation \eqref{quasi} are defined by holomorphic triangle counts with twisting on both the input and output complexes:
  $$  \underline{CF}^+(M_2;\RR[C_n])\underset{{\RR[C_n][[U]]} }{\otimes}\CFmct(\Tg{2},\widetilde{\Tg{2}})\overset{f_{\alpha,2,\tilde{2}}}\longrightarrow \underline{CF}^+(M_{\tilde{2}};\RR[C_n])$$  
 where the twisting on the input is induced by the basepoint $p\subset \gamma_g^2$, and on the output by $\tilde{p}\subset \tilde{\gamma}_g^2$.  Note that the complex associated to the pair $\Tg{2},\widetilde{\Tg{2}}$ is twisted by both basepoints, and thus is freely generated  over 
 $$\RR[[U]][C_n]\otimes_{\RR[[U]]} \RR[[U]][C_n].$$  Equivalently, we can think of it as a complex of  $\RR[[U]][C_n]-\RR[[U]][C_n]$ bimodules.  The boundary operator is given by
 $$ \partial(\zeta^i\x\tilde{\zeta}^j)= \sum_{\stackrel{\scriptstyle\phi\in \pi_2(\x,\y)}{{\mu(\phi)=1}}}\#{\widehat{\mathcal{M}}(\phi)}\cdot t^{\omega(\phi)}\cdot U^{n_w(\phi)}\cdot\zeta^{i+n_p(\partial\phi)}\y\tilde{\zeta}^{j+n_{\tilde{p}}(\partial\phi)},$$
where we use $\zeta$ (resp. $\tilde{\zeta}$) to record  the twisting induced by  $p$ (resp. $\tilde{p}$).   Its homology, viewed as either a right or left module over $\RR[[U]][C_n]$ can easily be computed:
 $$\underline{\bf HF}^-(\Tg{2},\widetilde{\Tg{2}})\cong (\RR[[U]][C_n]\oplus \RR[[U]][C_n])\otimes_\RR \Lambda^*(\RR^{g-1}),$$ 
 where a bimodule generator for the top dimensional summand  is given by 
 $$ \Theta_{2,\tilde{2}}= \sum_{i=1}^{n} \zeta^{-i}\theta_+ \tilde{\zeta}^i.$$
Here, $\theta_+$ is the explicit  $g$--tuple of intersection points representing the top-graded generator of the chain complex for $\#^{g}S^1\times S^2$ coming from the Heegaard diagram $(\Sigma,\GGG{2},\GGGt{2})$.   Now the map $f_{\alpha,2,\tilde{2}}$ is defined on generators by (we suppress the role of $U$):
 
 \begin{equation} f_{\alpha,2,\tilde{2}}( \x\zeta^i\otimes \zeta^j\y\tilde{\zeta^k}):=
 \sum_{\stackrel{\scriptstyle\psi\in \pi_2(\x,\y,\rrr)}{\stackrel{{\mu(\psi)=0}}{{{n_p(\partial\psi)=-i-j \ \text{mod}\  n}}}}}\#{\mathcal{M}}(\psi)\cdot t^{\omega(\psi)}\cdot \tilde{\zeta}^{k+n_{\tilde{p}}(\partial\psi)}\cdot\rrr
  \end{equation}
We wish to show that the  map $q_2$ defined in \eqref{quasi} induces an isomorphism on homology.  To do this, we observe 
\begin{equation}\label{eq:quad} \widehat{h}_{2,0,1,\tilde{2}}(\Theta_{2,0}\otimes \Theta_{0,1}\otimes \Theta_{1,\tilde{2}})= \sum_{i=1}^{n} \zeta^{n-i}\theta_+ \tilde{\zeta}^{i}=\Theta_{2,\tilde{2}},\end{equation}
Figure \ref{quads} and its caption explain the first equality, and for the second  we use the fact $\zeta^{n-i}=\zeta^{-i}$. 

  Next we note that 
$$  \widehat{f}_{\alpha,2,\tilde{2}}( \x\otimes 1\theta^+\tilde{\zeta}^j)=\tilde{\zeta}^j\iota(\x)+ \text{lower order terms},$$
where $\iota$ is the closest point map on generators, and lower order is with respect to the area filtration.  This follows  from the existence of small triangles connecting $\x$ to $\iota(\x)$ with third vertex mapping to $\theta_+$ whose boundaries do not cross the basepoints $p,\tilde{p}$.  Now consider the restriction of $q_2$ to the hat complex.  We have
$$
\begin{array}{lll} \widehat{q}_2(\zeta^j\x)&\!\! :=& 
\widehat{f}_{\alpha,2,\tilde{2}}( \x\zeta^j\otimes \Theta_{2,\tilde{2}})\\
&&\\
&=& \sum_{i=1}^{n} \widehat{f}_{\alpha,2,\tilde{2}}( \x\zeta^j\otimes \zeta^{-i}\theta_+ \tilde{\zeta}^i)\\ 
&&\\
&= & \sum_{i=1}^{n} \widehat{f}_{\alpha,2,\tilde{2}}( \x\ \ \otimes\zeta^{j-i}\theta_+ \tilde{\zeta}^i)\\ 
&&\\
&=& \widehat{f}_{\alpha,2,\tilde{2}}( \x\ \otimes 1\theta^+\tilde{\zeta}^j) + \text{lower order terms}\\
&&\\
&=&  \tilde{\zeta}^j\iota(\x)+ \text{lower order terms}.
\end{array}
$$
\noindent Thus $\widehat{q}_2$ is an isomorphism up to lower order terms which implies that it, and $q_2$, induce isomorphisms on homology.



\end{proof}

\bibliographystyle{plain}
\bibliography{mybib}

\end{document}